\documentclass[reqno, 12pt, a4paper]{amsart}  
\usepackage{articletemplate}
\usepackage{enumerate} 
\usepackage{xfrac}

\newcommand{\NoBlackBoxes}{\global\overfullrule0pt}
\NoBlackBoxes

\makeatletter
\AtBeginDocument{%
 \def\@serieslogo{%
 \vbox to\headheight{%
 \parindent\z@ \fontsize{6}{7\p@}\selectfont
June 13, 2014\endgraf
 \vss}}}
\makeatother

\numberwithin{equation}{section}

\newcommand{\tauconv}{\stackrel{\lower0.2ex\hbox{$\scriptscriptstyle
                    \it{\tau} $}}{\rightarrow}}

\newcommand{\Mane}{\textrm{Ma\~{n}\'{e}}}

\def\N{\mathbf{N}}

\def\R{\mathbf{R}}	
\def\Prob{P}

\theoremstyle{theorem}
\newtheorem{theorem}{Theorem}[section]

\newtheorem{proposition}[theorem]{Proposition}

\theoremstyle{definition}

\begin{document}

\title[Viscosity Solutions of Hamilton-Jacobi equations]{Min-max representations of viscosity solutions of Hamilton-Jacobi equations and applications in rare-event simulation}      
\renewcommand{\thefootnote}{\fnsymbol{footnote}} 
\footnotetext{\emph{MSC2010 subject classifications.} Primary 35D40, 35F21; secondary 65C05, 49L25.}     
\renewcommand{\thefootnote}{\arabic{footnote}}
\renewcommand{\thefootnote}{\fnsymbol{footnote}} 
\footnotetext{\emph{Key words and phrases.} Hamilton-Jacobi equations, rare-event simulation, importance sampling, variational problems.}     
\renewcommand{\thefootnote}{\arabic{footnote}}
   
\author[]{Boualem Djehiche \and Henrik Hult \and Pierre Nyquist}
\address[]{Department of Mathematics, KTH Royal Institute of Technology}
\maketitle   


\begin{abstract}
In this paper a duality relation between the $\Mane$ potential and Mather's action functional is derived in the context of convex and state-dependent Hamiltonians. The duality relation is used to obtain min-max representations of viscosity solutions of first order Hamilton-Jacobi equations.  These min-max representations naturally suggest class\-es of subsolutions of Hamilton-Jacobi equations that arise in the theory of large deviations. The subsolutions, in turn, are good candidates for designing efficient rare-event simulation algorithms.  
\end{abstract}

\section{Introduction} 

The motivation for this paper comes from the challenging problem to efficiently compute probabilities of rare events by stochastic simulation. Examples of such events include the probability that a diffusion process leaves a stable domain, voltage collapse in power systems, the probability of a large loss in a financial portfolio, the probability of buffer overflow in a queueing system, etc. See, e.g., \cite{AsmussenGlynn, Glasserman, RubinoTuffin} and references therein for numerous examples.   

For rare events the standard Monte Carlo technique fails because few particles will hit the rare event, leading to a large relative error. To reduce the variance a control mechanism that forces particles towards the rare event must be introduced. To obtain an unbiased estimator a weight is attached to each particle and the estimator is the sum of the weights of all the particles that end up in the rare event. The design of the controlled simulation algorithm must not only force particles towards the rare event,  but also keep the associated weights under control. Examples of such techniques include importance sampling and multi-level splitting, see \cite{AsmussenGlynn, RubinoTuffin},   as well as genealogical particle methods, see \cite{DelMoralGarnier}. 

Traditionally the design of rare event simulation algorithms are based on mimicking the large deviation behavior. More precisely, whenever a large deviation result is available that gives the exponential decay rate of probabilities of rare events and the most likely path to the rare event, the idea is to construct the control mechanism so that the system tends to follow the most likely path to the rare event. Although this approach has turned out to be reasonably successful, there are examples where this simple heuristic fail, see \cite{GlassermanWang}, and the design issue is delicate. More recently,  it has been demonstrated that, in many models in applied probability, the construction of efficient rare-event simulation algorithms is intimately connected with solutions to partial differential equations of Hamilton-Jacobi type that arise in large deviation theory. 

Suppose that the rate function associated with the large deviations of a sequence of stochastic processes $\{X^n(t); t\in [0 , T]\}$ is of the form
\begin{align*}
	\int_t^T \bar L(\psi(s), \dot \psi(s)) ds, 
\end{align*}
where $\psi$ is an absolutely continuous function and $\bar L$ is the local rate function, such that $v \mapsto \bar L(x,v)$ is convex for all $x \in \R^n$. Then, the large deviations rate of the probability $P_{t,x} (X^n(T) \notin \Omega )$, $0 \leq t < T$, $x \in \Omega$, where $\Omega$ is an open subset of $\R^n$, is given by
 \begin{align*}
	\bar U(t,x) = \inf_{\psi}\Big\{\int_t^T \bar L(\psi(s), \dot \psi(s)) ds, \psi(t) = x, \psi(T) \notin \Omega \Big\} 
\end{align*}
where the infimum is taken over all absolutely continuous functions. Since $\bar U$ is the value function of a variational problem it
satisfies a Hamilton-Jacobi terminal value problem  of the form
\begin{align}\label{eq:introUbar}
\begin{cases}	
	\bar U_t(t,x) - \bar H(x, -D\bar U(t,x)) = 0, & (t,x) \in [0,T) \times \Omega,\\
	\bar U(T,x) = 0, & x \in \partial \Omega,
\end{cases}
\end{align}
where $\bar H$ is the Fenchel-Legendre transform of $\bar L$, see e.g.\ \cite{FengKurtz}.
 	
In the context of  importance sampling the connection between efficient simulation algorithms and certain subsolutions of the Hamilton-Jacobi equation is established in \cite{DupuisWang, DW7, DSW, DLW3, VandenWeare}. See also \cite{DupuisDean, DupuisDean2} for multi-level splitting and \cite{DupuisCai} for genealogical particle methods. The essence of the developed theory is, roughly speaking, that the design of efficient stochastic simulation algorithms for computing probabilities of rare events is equivalent to finding subsolutions of the associated Hamilton-Jacobi equation whose value at the initial point agree with the value of the viscosity solution.  In this paper we develop a systematic approach to the construction of viscosity subsolutions, useful in rare-event simulation,  that are based on a novel min-max representation of viscosity solutions to the associated Hamilton-Jacobi equation. 

We consider Hamiltonians $(x,p) \mapsto H(x,p)$ that are convex in $p$ and satisfy standard continuity conditions. The main result, Theorem \ref{thm:duality}, proves a duality between \Mane's potential and Mather's action functional and is briefly described in what follows. 

With $L$ denoting the Fenchel-Legendre transform of $H$, the \Mane\ potential at level $c$ is given by
\begin{align*}
S^{c}(x, y) = \inf_{\psi, t}\Big\{ \int_{0}^{t} c + L(\psi(s), \dot \psi(s)) ds, \psi(0) = x, \psi(t) = y \Big\}, \quad x,y \in \R^n, 
\end{align*}
where the infimum is taken over all absolutely continuous functions $\psi: [0,\infty) \to \R^n$ and $t > 0$, see \cite{Mane}. Whenever it is finite, $y \mapsto S^c(x,y)$ is a viscosity subsolution of the stationary Hamilton-Jacobi equation
\begin{align*}
	H(y, DS(y)) = c, \quad y \in \R^{n},
\end{align*}
where $D$ denotes the gradient. Mather's action functional is given by\begin{align*}
	M(t,y; x) = \inf_{\psi}\Bigl\{ \int_{0}^{t} L(\psi(s), \dot \psi(s)) ds, \psi(0) = x, \psi(t) = y \Bigr\}, \quad t > 0 \; x,y \in \R^n,
\end{align*}
where the infimum is taken over all absolutely continuous functions $\psi: [0,t] \to \R^n$, see \cite{Mather}. From the variational representations it is elementary to show that
\begin{align*}
	S^{c}(x,y) = \inf_{t > 0} \{M(t,y;x) + ct\}. 
\end{align*}
The main result of this paper, Theorem \ref{thm:duality}, shows that the dual relation also holds:
\begin{align*}
	M(t,y; x) = \sup_{c > c_{H}}\{S^{c}(x,y) - ct\},
\end{align*}
where $c_{H}$ denotes the \Mane\ critical value; the infimum over $c \in \R$ for which the stationary Hamilton-Jacobi equation admits a global viscosity subsolution.  

From the duality result we derive min-max representations of viscosity solutions of various time-dependent problems. For the initial value problem
\begin{align*}
\begin{cases}
	V_{t}(t,y) + H(y, DV(t,y))  = 0, & (t,y) \in (0,\infty) \times \R^{n}, \\
	V(0,y) = g(y), & y \in \R^{n},  
\end{cases}
\end{align*}
we prove a min-max representation of the form
\begin{align*}
	V(t,y) = \inf_{x} \sup_{c > c_{H}} \{g(x) + S^{c}(x,y) - ct\}, \quad (t,y) \in [0,\infty) \times \R^{n}. 
\end{align*}
The min-max representation may be viewed as a generalization, to state-dependent Hamiltonians,  of the classical Hopf-Lax-Oleinik formula, which states that if $H(x,p) = H(p)$, then the solution to the initial value problem is given by
\begin{align*}
	V(t,y) = \inf_x\Big\{g(x) + tL\Big(\frac{y-x}{t}\Big)\Big\}.
\end{align*}
See \cite{Capuzzo, BardiEvans} for further details and generalizations of Hopf-Lax representation formulas to some state-dependent Hamiltonians. 

Similar min-max representations are stated for terminal value problems, problems on domains, and exit problems. For instance, the viscosity solution $\bar U$ to \eqref{eq:introUbar} can be represented as
\begin{align*}
	\bar U(t,x) = \inf_{y \in \partial \Omega} \sup_{c > c_{\bar H}} \{\bar S^c(x,y) - c(T-t)\},
\end{align*}
where $\bar S^c$ is the \Mane\ potential associated with $\bar L$. 

The min-max representations naturally suggest families of viscosity subsolutions useful for the design of rare-event simulation algorithms for time-dependent problems with state-dependent Hamiltonians. Indeed, for any $c > c_{\bar H}$, $y \in \partial \Omega$ and $K \geq 0$ sufficiently large, the function $(t,x) \mapsto \bar S^c(x,y) - c(T-t)-K$ is the type of subsolution to \eqref{eq:introUbar} that can be used to design efficient algorithms. We illustrate the applications in rare-event simulation in detail for exit problems of small-noise diffusions and birth-and-death processes. 

The paper is organized as follows. 
Background material on viscosity solutions of first order Hamilton-Jacobi equations is given in Section \ref{sec:visc}. The duality result is given in Section \ref{sec:Min-max} from which a min-max representation for the initial value problem is obtained. Similar representations for terminal value problems, problems on domains and exit problems are also presented. In Section \ref{sec:HopfLax} a direct relation between the min-max representation and the Hopf-Lax-Oleinik formula is presented for state-independent convex Hamiltonians. In Section \ref{sec:rare} it is shown how the min-max representation naturally suggests families of subsolutions appropriate for the design of efficient rare event simulation algorithms. Examples related to small-noise diffusions and birth-and-death processes are also provided.  

\section[Continuous viscosity solutions of Hamilton-Jacobi equations]{Continuous viscosity solutions of Hamilton-Jacobi equations}\label{sec:visc}
In this section a brief introduction to viscosity solutions of Hamilton-Jacobi equations is given. For more details the reader is referred to \cite{ CapDol,  Barles, Evans, Fathi, CrandallIshiiLionsU, CrandallEvansLions}. 

Suppose that the Hamiltonian $H: \R^{n} \times \R^{n} \to \R$ is convex in the second coordinate and  satisfies the following continuity condition:
 \begin{align}\label{eq:LipH}
 \left. \begin{array}{l}
  	\text{$H$ is uniformly continuous on $\R^n \times B_0(R)$ for each $R > 0$ and}  \\ 
	|H(x,p) - H(y,p)| \leq \omega(|x-y|(1+|p|)), \text{ for $x,y,p \in \R^{n}$}, \end{array} \right\}
\end{align}
where  $B_0(R) = \{p \in \R^n: |p| < R\}$ and $\omega: [0,\infty) \to [0,\infty)$ is a continuous nondecreasing function with $\omega(0) = 0$.

Given an initial function $g: \R^{n} \to \R$, the initial value problem for the Hamilton-Jacobi equation is to find $V: [0,\infty) \times \R^{n} \to \R$ satisfying
\begin{align}\label{eq:HJ}
	\begin{cases} V_{t}(t,x) + H(x, DV(t,x)) = 0, & (t,x) \in (0,\infty)  \times \R^{n}, \\
	V(0,x) = g(x), & x \in \R^{n}. 
	\end{cases}
\end{align}
where $V_t = \partial V/\partial t$ and $DV = (\partial V/\partial x_1, \dots, \partial V/\partial x_n)$.

In general it is impossible to find classical solutions to Hamilton-Jacobi equations.  Crandall and Lions have introduced the notion of viscosity solutions, see \cite{CrandallLions, CrandallEvansLions}.  
A continuous function $V : [0,\infty) \times \R^{n} \to \R$ is a \emph{viscosity subsolution (supersolution)} of \eqref{eq:HJ} if $V(0,x) \leq g(x)$ ($\geq g(x)$) and,  for every $v \in C^{\infty}((0, \infty) \times \R^{n})$,  
\begin{align*}
	 \left. \begin{array}{l}
	 \mbox{if $V -v$ has a local maximum (minimum) at $(t_{0}, x_{0}) \in (0,\infty) \times \R^{n}$,} \\ \text{then } v_{t}(t_{0}, x_{0}) + H(x_0, Dv(t_{0}, x_{0})) \leq 0 \quad (\geq 0).
	 \end{array}\right\}
\end{align*}
$V$ is a \emph{viscosity solution} if it is both a subsolution and a supersolution of \eqref{eq:HJ}. If the initial function $g$ is uniformly continuous  and $H$ satisfies \eqref{eq:LipH} then the comparison principle holds and the solution of the initial value problem \eqref{eq:HJ} is unique, see e.g.\  Theorem 3.7 and Remark 3.8 in Chapter II of \cite{CapDol}. 

Denote by $L$ the Fenchel-Legendre transform of $H$,  that is, 
\begin{align*}
	 L(x,v) &= \sup_{p} \{\langle p, v\rangle -  H(x,p)\}, \text{ and } \\ 
	 	 H(x, p) &= \sup_{v} \{\langle p, v\rangle -  L(x,v)\}.  
\end{align*}
Throughout the paper it will be assumed that 
\begin{align*}
\text{$L$ is continuous at $(x,0)$ for each $x \in \R^n$.} 
\end{align*}

Given a uniformly continuous function $g: \R^{n} \to \R$, let $V$ be the value function of the variational problem
\begin{align}\label{eq:V}
	V(t,y) = \inf_\psi \biggl \{ g(\psi(0)) + \int_{0}^{t} L(\psi(s), \dot \psi(s)) ds, 
	\psi(t) = y \biggr \}, 
\end{align}
where $(t,y) \in [0,\infty) \times \R^{n}$ and the infimum is taken over all absolutely continuous functions $\psi: [0,\infty) \to \R^{n}$.  It is well known that $V$ is the unique continuous viscosity solution to \eqref{eq:HJ}, see e.g.\  \cite[Ch.\ III, Sec.\ 3]{CapDol}. 


\subsection{The stationary Hamilton-Jacobi equation}

Given $c \in \R$, the stationary Hamilton-Jacobi equation is
\begin{align}\label{eq:sHJ}
	H(x, DS(x)) = c, \quad x \in \R^{n}. 
\end{align}
Similar to the time-dependent case, a continuous function $S: \R^n \to \R$ is a viscosity subsolution (supersolution) of the stationary Hamilton-Jacobi equation \eqref{eq:sHJ} if, for every function $v \in C^{\infty}(\R^{n})$,  
\begin{align}
	 \left.\begin{array}{l}
	 \mbox{if $S -v$ has a local maximum (minimum) at $x_{0} \in \R^{n}$, } \\ \text{then }  H(x_0, Dv(x_{0})) \leq c \quad (\geq c).
	 \end{array}\right\}  \label{eq:viscstat}
\end{align}
It is a viscosity solution if it is both a viscosity subsolution and a viscosity supersolution.  

The \emph{Ma\~n\'e critical value}, $c_{H}$,  is the infimum over $c$ for which \eqref{eq:sHJ} admits a viscosity subsolution. It may be observed that 
\begin{align}\label{eq:cH}
c_{H} \geq \sup_{x} \inf_{p} H(x,p). 
\end{align}
Indeed, if \eqref{eq:sHJ} admits a viscosity subsolution $U^c$ at level $c$, then for almost every $x$ there is a $v \in C^{\infty} (\R ^n)$ such that $U^c - v$ has a local maximum at $x$ and $\inf _p H(x,p) \leq H(x,Dv(x)) \leq c$. 
The claim follows by taking supremum over $x$.  Examples where $c_{H} = \sup_{x} \inf_{p} H(x,p)$ are provided below. 

For $c \in \R$, the \emph{Ma\~n\'e potential at level $c$}, originally introduced by \Mane\  in \cite{Mane}, is the function $S^{c}: \R^{n} \times \R^{n} \to \R$ defined by
\begin{align}\label{eq:Manepotential}
	S^{c}(x, y) = \inf_{\psi, t}\Big\{ \int_{0}^{t} c + L(\psi(s), \dot \psi(s)) ds, \psi(0) = x, \psi(t) = y \Big\}, \quad x,y \in \R^n, 
\end{align}
where the infimum is taken over all $t > 0$ and absolutely continuous $\psi: [0,\infty) \to \R^n$. It is useful to observe the following properties of $S^c$: for each $x \in \R^{n}$, $S^c(x,x) = 0$, for each $x, y \in \R^{n}$,  the function $c \mapsto S^{c}(x , y)$ is nondecreasing and  $S^{c}$ satisfies the triangle inequality:
\begin{align}\label{eq:tri}
	S^{c}(x, z) \leq S^{c}(x, y) + S^{c}(y,z), \quad x,y,z \in \R^n. 
\end{align}
The property $S^c(x,x) = 0$ follows from the triangle inequality. 
To prove the triangle inequality, take an arbitrary $\epsilon > 0$, and select $t_1, t_2 > 0$ and absolutely continuous functions $\psi_1, \psi_2$ with $\psi_1(0) = x$, $\psi_1(t_1) = y$, $\psi_2(0) = y$ and $\psi_2(t_2) = z$ such that
\begin{align*}
	S^c(x,y) &\geq \int_0^{t_1} c + L(\psi_1(s), \dot \psi_1(s)) ds - \frac{\epsilon}{2}, \\
	S^c(y,z) &\geq \int_0^{t_2} c + L(\psi_2(s), \dot \psi_2(s)) ds - \frac{\epsilon}{2}. 
\end{align*}
Concatenate the two trajectories by
\begin{align*}
	\psi(s) = \psi_1(s)I\{0 \leq s \leq t_1\} + \psi_2(s-t_1)I\{t_1 < s \leq t_1+t_2\}. 
\end{align*}
It follows that
\begin{align*}
	S^c(x,y) + S^c(y,z) &\geq \int_0^{t_1} c + L(\psi_1(s), \dot \psi_1(s)) ds \\
	& \quad  +  \int_0^{t_2} c + L(\psi_2(s), \dot \psi_2(s)) ds  - \epsilon  \\
	&= \int_0^{t_1 + t_2} c + L(\psi(s), \dot \psi(s)) ds  - \epsilon  \\
	& \geq S^c(x,z) - \epsilon. 
\end{align*}
Since $\epsilon >0$ is arbitrary the triangle inequality follows. 

It is possible that $S^{c}$ is identically  $-\infty$ for small $c$. Indeed, if $L(x,v) = \frac{1}{2}|v|^2$ and $c < 0$, then it follows from the variational representation \eqref{eq:Manepotential} that $S^c(x,y) = -\infty$ for all $x,y \in \R^n$.

The \Mane\ potential is well studied within weak KAM theory where it is commonly assumed that the Hamiltonian is uniformly superlinear; for each $K  \geq 0$ there exists $C^*(K) \in \R$ such that $H(x,p) \geq K|p| - C^*(K)$ for each $x,p$. Under such an assumption there exist critical viscosity subsolutions, that is, there exists a global viscosity subsolution to \eqref{eq:sHJ} for $c = c_H$, see \cite{FathiMaderna, Fathi}. In this paper it is only assumed that the Hamiltonian is convex in $p$, see \eqref{eq:LipH}. For instance, the Hamiltonian associated with the unit rate Poisson process, which is of the form
\begin{align*}
	H(p) = e^p-1, \quad p \in \R,
\end{align*}
is covered by our assumptions. For this choice of $H$ the \Mane\ critical value is $c_H = -1$, but there can be no critical subsolution $S$ as it would have to satisfy $DS(x) = -\infty$ almost eveywhere. 


The following properties of the \Mane\ potential are well known and similar statements appear in \cite{Fathi, FathiMaderna, FathiSiconolfi}, see also the lecture notes \cite{FathiLN, BernardLN}. Because our assumptions on the Hamiltonian are slightly different a proof is included for completeness. 

\begin{prop}\label{prop:Fathi} Let $c \in \R$. 
	\begin{enumerate}[(i)]
	\item Suppose that $S^c > -\infty$. For each $x \in \R^{n}$ the function $y \mapsto S^c(x,y)$ is a viscosity subsolution to $H(y, DS(y)) = c$ on $\R^n$ and a viscosity solution on $\R^n \setminus \{x\}$. 
	\item $S^c(x,y) = \sup_{S \in \mathcal{S}^c_x} S(y)$, for each $x,y \in \R^n$,  where $\mathcal{S}^c_x$ is the collection of all continuous viscosity subsolutions to $H(y, DS(y)) = c$ that vanish at $x$. 
	\end{enumerate}
\end{prop}

We conclude that for $c > c_H$ there exist viscosity subsolutions to \eqref{eq:sHJ} and by Proposition \ref{prop:Fathi}(ii) it follows that $S^c > - \infty$ . Similarly, for $c < c_H$ there are no subsolutions and by Proposition \ref{prop:Fathi}(i) $S^c = -\infty$. 

Before proceeding to the proof of Proposition \ref{prop:Fathi} we state an important lemma that can be interpreted as a dynamic programming property of the \Mane\ potential. 
\begin{lem}\label{lem:dp}
	Suppse that $S^c > -\infty$. For any $x,y_0 \in \R^n$ with $y_0 \neq x$ and $\epsilon > 0$ there exist $0 < \delta < |x-y_0|$, $y$ with $|y-y_0| < \delta$, $h > 0$ and  an absolutely continuous  function $\psi$ with $\psi(0) = y$, $\psi(h) = y_0$, and $|\psi(s) - y_0| < \delta$ for all $s \in [0,h]$, such that
\begin{align*}
	S^c(x,y_0) \geq S^c(x,y) + \int_0^h c + L(\psi(s), \dot \psi(s)) ds - \epsilon. 
\end{align*}
\end{lem}

\begin{proof}
Given $x,y_0 \in \R^n$ with $x \neq y_0$ and  $\epsilon > 0$,  take $t > 0$ and an absolutely continuous function $\varphi$ with $\varphi(0) = x$, $\varphi(t) = y_0$ such that
\begin{align*}
	S^c(x,y_0) \geq \int_0^t c+L(\varphi(s), \dot \varphi(s)) ds - \epsilon. 
\end{align*}
Let $0 < \delta < |x-y_0|$ and take  $h > 0$ such that $|\varphi(s) - y_0| < \delta$ for each $s \in [t-h,t]$. With $y = \varphi(t-h)$ and $\psi(s) = \varphi(s+t-h)$, $s \in [0,h]$,  it follows that
\begin{align*}
	S^c(x,y_0) &\geq \int_0^t c+L(\varphi(s), \dot \varphi(s)) ds - \epsilon \\
	&= \int_0^{t-h} c+L(\varphi(s), \dot \varphi(s)) ds  + \int_{t-h}^t c + L(\varphi(s), \dot \varphi(s)) ds - \epsilon \\
	&\geq S^c(x,y) + \int_0^h c+L(\psi(s), \dot \psi(s)) ds - \epsilon.
\end{align*}
This completes the proof.
\end{proof}

\begin{proof}[Proof of Proposition \ref{prop:Fathi}]
Proof of (i).  Suppose that $S^c > -\infty$ and take $x \in \R^n$.  First we prove the viscosity subsolution property. 
For $v \in C^\infty(\R^n)$, suppose that $S^c(x,\cdot)-v$ has a local maximum at $y_0$ and, contrary to what we want to show, that $H(y,Dv(y)) -c \geq \theta > 0$ for $|y-y_0| \leq \delta$, for some $\delta > 0$. We may assume that $\delta$ is sufficiently small that 
\begin{align*}
	S^c(x,y) - v(y) \leq S^c(x, y_0) - v(y_0), \quad \text{for } |y-y_0| \leq \delta.
\end{align*}
Take any $y$ with $|y-y_0| \leq \delta$ and consider any absolutely continuous function $\psi$ such that $\psi(0) = y$, $\psi(h) = y_0$ and $|\psi(s) - y_0| \leq \delta$ for all $s \in [0,h]$. By the triangle inequality \eqref{eq:tri} and the last inequality
\begin{align*}
	0 &\geq S^c(x, y_0) - S^c (x,y) - \int _{0} ^ {h} c+L(\psi(s), \dot \psi(s))ds \\
	& \geq v(y_0) - v(y) - \int _{0} ^ {h} c+ L(\psi(s), \dot \psi(s))ds \\
	&= \int _{0} ^{h} \frac{d}{ds} v (\psi (s)) - L(\psi(s), \dot \psi(s)) -c \;ds \\
	& = \int _{0} ^{h} \langle Dv(\psi(s)), \dot \psi (s) \rangle - L(\psi(s), \dot \psi(s))  -c\; ds.
\end{align*}
We may assume that $\dot \psi$ is chosen such that, using the conjugacy between $H$ and $L$,
\begin{align*}
	H(\psi(s) , Dv(\psi(s))) \leq \langle Dv(\psi(s)), \dot \psi(s) \rangle - L(\psi(s), \dot \psi(s)) + \frac{\theta}{2},
\end{align*}
for all $s\in [0,h]$. Then
\begin{align*}
	\frac{\theta h}{2} \geq \int _{0} ^{h} H(\psi(s), Dv(\psi(s))) -c \;ds \geq \theta h,
\end{align*}
which is a contradiction. Thus, it must indeed hold that $H(y_0, Dv(y_0)) \leq c$.

Next,  we prove the supersolution property on $\R^n \setminus \{x\}$. Take $v \in C^\infty(\R^n)$ and suppose $S^c(x,\cdot)-v$ has a local minimum at $y_0 \neq x$ and, contrary to what we want to show, that $H(y,Dv(y)) - c \leq -\theta < 0$ for $|y-y_0| \leq \delta$, for some $\delta > 0$. We may assume that $\delta$ is sufficiently small that $|x-y_0| > \delta$ and
\begin{align*}
	S^c(x,y) - v(y) \geq S^c(x,y_0) - v(y_0), \quad \text{for } |y-y_0| \leq \delta.
\end{align*}

By Lemma \ref{lem:dp} we may select $y$ with $|y-y_0| \leq \delta$ and an absolutely continuous $\psi$ such that $\psi(0) = y$, $\psi(h) = y_0$ and $|\psi(s) - y_0| \leq \delta$ for all $s \in [0,h]$, with the property that
\begin{align*}
	S^c(x, y_0) \geq S^c(x,y) + \int_0^h c+L(\psi(s), \dot \psi(s)) ds - \frac{\theta h}{2}.
\end{align*}
The last inequality implies that
\begin{align*}
	\frac{\theta h}{2} &\geq S^c(x,y) - S^c (x, y_0) + \int _{0} ^ {h} c+ L(\psi(s), \dot \psi(s))ds \\
	& \geq v(y) - v(y_0) + \int _{0} ^ {h} c+ L(\psi(s), \dot \psi(s))ds \\
	&= \int _{0} ^{h} - \frac{d}{ds} v (\psi (s)) + L(\psi(s), \dot \psi(s)) + c\;ds\\
	& = \int _{0} ^{h} - \langle Dv(\psi(s)), \dot \psi (s) \rangle + L(\psi(s), \dot \psi(s)) +c\; ds \\
	&\geq \int_0^h - \Big( H(\psi(s), Dv(\psi(s))) - c \Big) ds. 
\end{align*}
We conclude that
\begin{align*}
	- \frac{\theta h}{2} \leq \int _{0} ^{h} H(\psi(s), Dv(\psi(s)))  - c \;ds \leq  -\theta h,
\end{align*}
which is a contradiction. Thus, it must indeed hold that $H(y_0, Dv(y_0)) \geq c$. This completes the proof of (i).

Proof of (ii). Let $c \in \R$. If there are no viscosity subsolutions at level $c$, then by (i) $S^c = -\infty$ and $\mathcal{S}^c_x = \emptyset$, which implies that $\sup_{S \in \mathcal{S}^c_x} S(y) = -\infty$ as well. 
If there exist continuous viscosity subsolutions at level $c$, take $x \in \R^n$ and let $S$ be a continuous viscosity subsolution of $H(y, DS(y)) = c$ on $\R^n$. It is sufficient to show that for any $y \in \R^n$, $t > 0$ and absolutely continuous function $\psi$ with $\psi(0) = x$ and $\psi(t) = y$, 
\begin{align} \label{eq:leqint}
	S(y) - S(x) \leq \int_0^t c + L(\psi(s), \dot \psi(s)) ds. 
\end{align}

To show \eqref{eq:leqint}, fix $t>0$, $y\in \R ^n$, an absolutely continuous $\psi$ with $\psi(0)=x$ and $\psi(t) = y$ and take an arbitrary $\epsilon > 0$. For every $s \in [0,t]$, let $v_s \in C^\infty(\R^n)$ be such that $S-v_s$ has a local maximum at $\psi(s)$. Then, there exists $\delta_s > 0$ such that 
\begin{align*}
	S(z) - v_s(z) \leq S(\psi(s)) - v_s(\psi(s)), \quad \text{ for } |z-\psi(s)| < \delta_s, 
\end{align*}	
and consequently that 
\begin{align}
	S(z) - S(\psi(s)) \leq v_s(z) - v_s(\psi(s)), \quad \text{ for } |z-\psi(s)| < \delta_s. \label{eq:t1}
\end{align}
By continuity of $H$ and $Dv_s$ we may, in addition,  assume that $\delta_s$ is sufficiently small that
\begin{align*}
	H(z, Dv_s(z)) \leq c + \frac{\epsilon}{t}, \quad \text{ for } |z-\psi(s)| < \delta_s.
\end{align*}
For every $s \in [0,t]$, let $h_s > 0$ be such that $|\psi(u) - \psi(s)| < \delta_s$ for every $u$ with $|u-s| < h_s$. This is possible due to the continuity of $\psi$. The union
\begin{align*}
	[0, h_0) \cup \bigcup_{s \in (0,t]} (s, s+h_s), 
\end{align*}
is an open cover of  $[0,t]$. Since $[0,t]$ is compact there is a finite subcover, which we may assume is of the form
\begin{align*}
	[0, h_0) \cup \bigcup_{k=1}^{n-1} (s_k, s_k + h_{s_k}), 
\end{align*}
where $0 = s_0 < s_1 < \dots < s_{n-1} < s_{n} = t$.  Since the finite union is a 
subcover, it must hold that $s_{k-1} < s_k < s_{k-1} + h_{s_{k-1}}$ for each $k = 1,\dots, n$. It follows that,  using \eqref{eq:t1} and the conjugacy between $H$ and $L$,
\begin{align*}
 S(y) - S(x) &= \sum_{k=1}^n S(\psi(s_k)) - S(\psi(s_{k-1})) \\
 & \leq \sum_{k=1}^n v_{s_{k-1}}(\psi(s_k)) - v_{s_{k-1}}(\psi(s_{k-1})) \\
 &= \sum_{k=1}^n \int_{s_{k-1}}^{s_k} \langle Dv_{s_{k-1}}(\psi(s)), \dot \psi(s) \rangle ds \\ 
 &\leq \sum_{k=1}^n \int_{s_{k-1}}^{s_k}  H(\psi(s), Dv_{s_{k-1}}(\psi(s))) + L(\psi(s), \dot \psi(s))\, ds \\ 
 &\leq \sum_{k=1}^n \int_{s_{k-1}}^{s_k}  c+\frac{\epsilon}{t} +  L(\psi(s), \dot \psi(s))\, ds \\ 
&= \epsilon  + \int_0^t c + L(\psi(s), \dot \psi(s))\, ds. 
\end{align*}
Since $\epsilon > 0$ was arbitrary the claim follows. 
\end{proof}

We proceed by computing \Mane's critical value, $c_H$, for some Hamiltonians arising in the theory of large deviations of stochastic processes. 

\begin{exmp}[Critical diffusion process]\label{ex:diffusion}
Let $U:\R^{n} \to \R$ be a potential function and $b(x) = -DU(x)$. Consider the Hamiltonian $H(x,p) = \langle b(x), p\rangle + \frac{1}{2}|p|^2$. Then $c_{H} = \sup_{x} \inf_{p} H(x,p) = -\frac{1}{2}\inf_{x}|b(x)|^{2}$. Indeed, from \eqref{eq:cH}, $c_{H} \geq  -\frac{1}{2}\inf_{x}|b(x)|^{2}$ and $U$ is a subsolution to $H(x,DS(x)) = -\frac{1}{2}\inf_{x}|b(x)|^{2}$, which implies $c_{H} \leq -\frac{1}{2}\inf_{x}|b(x)|^{2}$. In particular, if $DU(x) = 0$ for some $x$, then $c_{H} = 0$.  The \Mane\ potential can be viewed as a generalization of Freidlin-Wentzell's quasi-potential described in \cite[Ch.\ 4]{Freidlin1984}. 
\end{exmp}

\begin{exmp}[Birth-and-death process]\label{ex:birthdeath}
Consider an interval $(a,b) \subset \R$ and functions $\mu: (a,b) \to [0,\infty)$ and $\lambda: (a,b) \to [0,\infty)$ satisfying $\int_a^b \log(\sqrt{\mu(x)/\lambda(x)}) dx < \infty$. Consider the Hamiltonian
\begin{align*}
	H(x,p) = \lambda(x)(e^p -1)+ \mu(x)(e^{-p}-1). 
\end{align*}
In this case $c_H = \sup_x \inf_p H(x,p) = -\inf_x (\sqrt{\mu(x)}-\sqrt{\lambda(x)})^2$. To see this, recall from \eqref{eq:cH} that $c_H \geq  -\inf_x (\sqrt{\mu(x)}-\sqrt{\lambda(x)})^2$. A subsolution of 
$$H(x,DS(x)) = -\inf_x (\sqrt{\mu(x)}-\sqrt{\lambda(x)})^2,$$
is given by
\begin{align*}
	U(x) = \int_a^x \log(\sqrt{\mu(z)/\lambda(z)}) dz. 
\end{align*}
Indeed, 
\begin{align*}
	H(x,DU(x)) = -(\sqrt{\mu(x)}-\sqrt{\lambda(x)})^2 \leq -\inf_x (\sqrt{\mu(x)}-\sqrt{\lambda(x)})^2. 
\end{align*}
\end{exmp}

\begin{exmp}[Pure birth process]\label{ex:pureBirth}
Let $\lambda: [0,\infty)^{n} \to [0,\infty)^{n}$ and put  
\begin{align*}
	H(x,p) = \sum_{j=1}^{n} \lambda_{j}(x)(e^{p_{j}} -1). 
\end{align*}
In this case $c_{H} = \sup_{x} \inf_{p} H(x,p) = - \inf_{x} \sum_{j=1}^{n} \lambda_{j}(x) =: -\lambda_{*}$. 
Indeed, from \eqref{eq:cH} it follows that $c_{H} \geq  -\lambda_{*}$ and for any $c \in (-\lambda_{*} , 0)$ and $\alpha \leq \log(1+c/\lambda_{*})$, the function $\alpha \langle 1, x\rangle$ is a subsolution to $H(x,DS(x)) = c$, which implies $c_{H} \leq -\lambda_{*}$.
\end{exmp}

\section{Duality and min-max representations} \label{sec:Min-max}
In this section a duality result is presented from which min-max representations of viscosity solutions are obtained. Min-max representations are formulated for  initial value problems,  terminal value problems, problems on domains,  as well as exit problems. 

\subsection{Duality}

Let us consider a Hamiltonian $H$ satisfying \eqref{eq:LipH}. As in the previous section the \Mane\ potential is denoted by $S^{c}$. By Proposition \ref{prop:Fathi}(i), $y \mapsto S^{c}(x,y)$ is a viscosity subsolution to  $H(y, DS(y)) = c$ for each $x \in \R^{n}$ and $c > c_{H}$. It follows immediately that the function
$(t,y) \mapsto S^{c}(x,y) - ct$ is a viscosity subsolution of the evolutionary Hamilton-Jacobi equation 
\begin{align}\label{eq:eHJ}
V_{t}(t,y) + H(y, DV(t,y)) = 0, \quad (t,y) \in (0,\infty) \times \R^n. 
\end{align}
For any $x \in \R^n$, Perron's method, see \cite[Theorem V.2.14]{CapDol}, implies that the function $U(\cdot\, ; x)$ given by  
\begin{align*}
	U(t,y; x) = \sup_{c > c_{H}}\{S^{c}(x,y) - ct\}, \quad (t,y) \in [0,\infty) \times \R^n
\end{align*}
is also a viscosity subsolution to \eqref{eq:eHJ} for any $x \in \R^n$. Moreover,  
$y \mapsto S^{c}(x,y)$ is a viscosity solution to $H(y, DS(y)) = c$ on $\R^{n} \setminus \{x\}$. This property also transfers to $U(t,y; x)$ as the following proposition shows. 

\begin{prop}\label{prop:auxprop}
For each $x \in \R^{n}$,  $U(t,y; x) = \sup_{c > c_{H}} \{S^{c}(x,y)-ct\}$ is a viscosity solution to \eqref{eq:eHJ} on $(0,\infty) \times \R^n\setminus\{x\}$ .  
\end{prop}

\begin{proof} Since $y\mapsto S^{c}(x,y)$ is a viscosity subsolution to $H(y,DS(y)) = c$, for any $c > c_{H}$, it follows by Perron's method that $U(t,y;x)$ is a viscosity subsolution to \eqref{eq:eHJ}. It remains to show the supersolution property.

Fix $x \in \R^{n}$ and  take $v \in C^{\infty}((0,\infty)\times \R^{n}\setminus \{x\})$. Suppose that $U(\cdot\, ;x) - v$ has a local minimum at $(t_{0}, y_{0})$ where $y_{0} \neq x$. We must show that $v_{t}(t_{0}, y_{0}) + H(y_{0}, Dv(t_{0}, y_{0})) \geq 0$. 

Suppose, on the contrary, that there exists a $\theta > 0$ such that 
\begin{align*}
	v_{t}(t, y) + H(y, Dv(t, y)) \leq - \theta, 
\end{align*}
for all $(t,y)$ with $|t-t_{0}| + |y-y_{0}| < \delta$ for some $\delta > 0$. We may assume that $\delta$ is sufficiently small that $|x-y_0| > \delta$ and 
\begin{align*}
	U(t,y;x)-v(t,y) \geq U(t_{0}, y_{0};x) - v(t_{0}, y_{0}), 
\end{align*}
for all $(t,y)$ with $|t-t_{0}| + |y-y_{0}| < \delta$. For all absolutely continuous $\psi$ with $\psi(0) = y$, $\psi(h) = y_{0}$, $0 < h < \delta - |y-y_0|$, such that $s + |\psi(s) - y_0| < \delta$ for all $s \in [0,h]$, the previous inequality, with $t = t_0 - h$, implies that  
\begin{align*}
	&U(t_0, y_0; x) - U(t_0-h,y; x) - \int_0^h L(\psi(s), \dot\psi(s))ds  \\ & \quad 
		\leq v(t_0, y_0) - v(t_0-h,y) - \int_0^h L(\psi(s), \dot\psi(s))ds  \\ &\quad = \int_0^h \frac{d}{ds} v(s,\psi(s))  - L(\psi(s), \dot\psi(s)) ds
	\\ &\quad 
	 =  \int_0^h v_t(s,\psi(s)) + \langle Dv(s,\psi(s)), \dot \psi(s)\rangle - L(\psi(s), \dot\psi(s))  ds
	\\ & \quad \leq \int_0^h  v_t(s,\psi(s)) + H(\psi(s), Dv(s,\psi(s))) ds \\ &\quad \leq -\theta h. 
\end{align*}
Take $c > c_H$ such that
\begin{align*}
 	U(t_0-h,y;x)  \leq S^c(x,y) - c(t_0-h) + \frac{\theta h}{2}. 
\end{align*}
Combining the last two displays shows that
\begin{align*}
	-\theta h &\geq U(t_0, y_0; x) - U(t_0-h,y; x) - \int_0^h L(\psi(s), \dot\psi(s))ds \\ & \geq U(t_0, y_0; x) - \Big(S^c(x,y) - c(t_0-h) + \frac{\theta h}{2}\Big) - \int_0^h L(\psi(s), \dot\psi(s))ds \\
	& \geq  S^c(x,y_0) - \Big(S^c(x,y) +  \int_0^h c+ L(\psi(s), \dot\psi(s))ds + \frac{\theta h}{2} \Big)
\end{align*}
and we conclude that
\begin{align*}
	S^c(x,y_0) \leq S^c(x,y) +  \int_0^h c+ L(\psi(s), \dot\psi(s))ds - \frac{\theta h}{2}. 
\end{align*}
This contradicts the statement of Lemma \ref{lem:dp} and completes the proof. 
\end{proof}

For any $x \in \R^n$ and $(t,y) \in (0,\infty) \times \R^n$, let
\begin{align*}
	M(t,y; x) = \inf_{\psi}\Big\{\int_0^t L(\psi(s), \dot \psi(s)) ds, \psi(0) = x, \psi(t) = y\Big\},
\end{align*}
where the infimum is taken over all absolutely continuous $\psi: [0,\infty) \to \R^n$. $M$ is Mather's action functional, see \cite{Mather},  viewed as a function of $(t,y)$. 

\begin{prop}\label{prop:W} Let $x \in \R^n$. 
\begin{enumerate}[(i)]
\item  $M(\cdot\,; x)$ is a viscosity subsolution to \eqref{eq:eHJ} on $(0,\infty) \times \R^n$ and a viscosity solution on $(0,\infty) \times \R^n\setminus \{x\}$. 
\item  $M(t,y; x) = \sup_{V \in \mathcal{S}_{0,x}} V(t,y)$, where $\mathcal{S} _{0,x}$ is the collection of all 
	continuous viscosity subsolutions to  \eqref{eq:eHJ} vanishing at $(0,x)$. 
	\end{enumerate}
\end{prop}

The proof is identical  to that of Proposition \ref{prop:Fathi} except for minor notational differences. The details are provided in the Appendix.

From the variational representation \eqref{eq:Manepotential} of the \Mane\ potential it follows immediately that
\begin{align*}
	S^c(x,y) = \inf_{t > 0} \{M(t,y;x) + ct\}. 
\end{align*}
The dual relationship also holds.
 
\begin{theorem}[Duality]\label{thm:duality}
	For each $x,y \in \R^n$, 
	\begin{align}
	S^c(x,y) &= \inf_{t > 0} \{M(t,y;x) + ct\}, \label{eq:duality1} \\
	M(t,y; x) &= \sup_{c > c_H}\{S^c(x,y) - ct\}. \label{eq:duality2} 
	\end{align}
\end{theorem}

\begin{proof}
 As mentioned above \eqref{eq:duality1} follows from the variational representation \eqref{eq:Manepotential} of the \Mane\ potential. 

Let us prove \eqref{eq:duality2}. Let $U(t,y; x) =  \sup_{c > c_H}\{S^c(x,y) - ct\}$. 
It follows from \eqref{eq:duality1} that $U(t,y;x) \leq M(t,y;x)$ because
\begin{align*}
	U(t,y; x) =  \sup_{c > c_H}\inf_{s>0}\{M(s,y;x) + c(s-t)\} \leq M(t,y;x). 
\end{align*}
The reverse inequality is proved next. By Proposition \ref{prop:auxprop},  it follows that $U(\cdot\,;x)$ is a viscosity supersolution to \eqref{eq:eHJ} on $(0,\infty) \times \R^n\setminus\{x\}$. By Proposition \ref{prop:W}(i), $M(\cdot\,; x)$ is a viscosity subsolution to \eqref{eq:eHJ}. Since $U(0,x;x) = 0 = M(0,x;x)$ the comparison principle implies that $U(t,y; x) \geq M(t,y;x)$ for all $(t,y) \in (0,\infty) \times \R^n\setminus \{x\}$. This shows \eqref{eq:duality2} for $y \neq x$. It remains to show the inequality for $y=x$, that is that for all $(t,x) \in (0,\infty) \times \R^n$, $M(t,x;x) \leq U(t,x;x)$. 

It follows from the variational representation of $M$ that, for any $y \in \R^n \setminus \{x\}$ and $h \in (0,t)$, 
\begin{align*}
	M(t,x; x) \leq M(t-h,y;x) + M(h,x;y).  
\end{align*}
By the duality result for $y \neq x$ it follows that
\begin{align*}
	M(t-h,y;x) + M(h,x;y) = U(t-h,y;x) + U(h,x;y). 
\end{align*}
The proof is completed by showing that, for any $\epsilon > 0$, we may select $h > 0$ and $y \neq x$ such that 
\begin{align*}
	U(t-h,y;x) + U(h,x;y) \leq U(t,x;x) + \epsilon. 
\end{align*}
To achieve this, take $0 < h < \min \{ 1, t/2 \}$ such that
\begin{align*}
	2h\bigl(L(x,0) + \frac{3}{2}+c_H\bigr) \leq \epsilon. 
\end{align*}
By continuity of $L$ we may select $\delta > 0$ such that 
\begin{align*}
	L(x+z,v) \leq L(x,0) + 1,
\end{align*}
for all $|z| \leq \delta$ and $|v| \leq \delta$. 

Take $y$ such that $h^{-1}|y-x| \leq \delta$. With $\psi(0) = x$, $\psi(h) = y$ and $\dot \psi(s) = h^{-1}(y-x)$ it follows from the variational representation of the \Mane\ potential that
\begin{align}\label{eq:Sxybound}
	S^c(x,y) \leq \int_0^h  c + L\Big(x + \frac{y-x}{h}s, \frac{y-x}{h}\Big) ds \leq  h(c+L(x,0) + 1). 
\end{align}
Similarly, 
\begin{align}\label{eq:Syxbound}
	S^c(y,x) \leq \int_0^h  c + L\Big(y + \frac{x-y}{h}s, \frac{x-y}{h}\Big) ds \leq  h(c+L(y,0) + 1). 
\end{align}
As a consequence of \eqref{eq:Sxybound}, 
\begin{align*}
	U(t-h,y;x) - U(t,x;x) &= \sup_{c > c_H}\{S^c(x,y) - c(t-h)\} + c_H t \\
	&\leq \sup_{c > c_H}\{ h(L(x,0) + 1) + c(2h-t)\} + c_H t \\
	&= h(L(x,0) + 1 + 2c_H). 
\end{align*}
Similarly, by \eqref{eq:Syxbound},
\begin{align*}
	U(h,x;y)  &= \sup_{c > c_H}\{S^c(y,x) - ch\} 
	\leq \sup_{c > c_H}\{ h(L(y,0) + 1)\} = h(L(y,0) + 1),
\end{align*}
and since $|y-x| \leq h \delta \leq \delta$, $h(L(y,0)+1) \leq h(L(x,0) + 2)$.
Combining the two inequalities shows that
\begin{align*}
	U(t-h,y;x) + U(h,x;y) &\leq U(t,x;x) + h(L(x,0) + 1 + 2 c_H) + h(L(x,0) + 2) \\
	&= U(t,x;x) + 2h\bigl(L(x,0) + \frac{3}{2} + c_H\bigr) \\
	& \leq U(t,x;x) + \epsilon, 
\end{align*}
by the choice of $h$. This completes the proof.
\end{proof}

The duality between $S^c(x,y)$ and $M(t,y; x)$ can be given the following intuitive physical interpretation. The optimal $t$ in the representation \eqref{eq:duality1} is the optimal time it takes to move from $x$ to $y$ in a system with energy level $c$. Similarly, the optimal $c$ in the representation \eqref{eq:duality2} is the energy level at which it takes precisely time $t$ to move from $x$ to $y$ along the most cost efficient path.

\subsection{Initial value problems}

Let $V$ be defined by \eqref{eq:V} where the initial function $g$ is uniformly continuous.  From Theorem \ref{thm:duality} the following min-max representation of $V$ is obtained. 
 
 \begin{cor}\label{thm:min-max} For all  $(t,y) \in [0,\infty) \times \R^n$, 
\begin{align}
V(t,y)  &=  \inf_{x} \sup_{c > c_H} \{ g(x) + S^c(x,y) - ct\}. \label{eq:min-max}
\end{align}
\end{cor}

\begin{proof} 
For all  $(t,y) \in [0,\infty) \times \R^n$, it follows from \eqref{eq:duality2} that
\begin{align*}
V(t,y)  &=  \inf_x\{g(x) + M(t,y; x)\} =  \inf_{x} \sup_{c > c_H} \{ g(x) + S^c(x,y) - ct\}.
\end{align*}
\end{proof}

\subsection{Terminal value problems} 

Let the following be given: a time $T > 0$, a Hamiltonian $\bar H$ satisfying \eqref{eq:LipH}, the associated Lagrangian $\bar L$ given by $\bar L(x,v) = \sup_p \{\langle p,v\rangle - \bar H(x,p)\}$, and a uniformly continuous terminal cost function $g$. Consider a terminal value problem with value function, for $(t,x) \in [0,T] \times \R^{n}$,
\begin{align*}
	\bar V(t,x) = \inf_{\psi}\left\{ \int_{t}^{T} \bar L(\psi(s), \dot \psi(s)) ds + g(\psi(T)), 
	\psi(t) = x \right\}, 
\end{align*} 
 where the infimum is taken over all absolutely continuous functions $\psi: [0,T] \to \R^n$ with $\psi(t) = x$. By changing the direction of the paths it follows that $\bar V(t,x)$ is equal to
 \begin{align*}
	\inf\left\{ g(\psi(0)) + \int_{0}^{T-t} \bar L(\psi(s), - \dot \psi(s)) ds,   
	\psi(T-t) = x \right\} = V(T-t, x),
\end{align*}
where $V$ is the value function of the forward problem \eqref{eq:V} with $L(x,v) = \bar L(x,-v)$. The Hamiltonian of the corresponding forward problem is
\begin{align*}
	H(x,p) = \sup_v \{\langle p,v\rangle - L(x,v)\} = \sup_v \{\langle -p,-v\rangle - \bar L(x,-v)\} = \bar 	H(x,-p).
\end{align*}
Since $V$ is the unique continuous viscosity solution to \eqref{eq:HJ} it follows that $\bar V$ is the unique continuous viscosity solution to 
\begin{align*}
\begin{cases}
 \bar V_{t}(t,x) - \bar H(x, -D\bar V(t,x)) = 0, & (t,x) \in [0,T)  \times \R^{n}, \\
	\bar V(T,x) = g(x), & x \in\R^{n}. 
\end{cases}
\end{align*}
For $c > c_{\bar H}$, let $\bar S^c(x,y)$ denote the \Mane\ potential associated with $\bar L$. Then, it holds that $\bar S^c(x,y) = S^c(y,x)$ and the min-max representation of Corollary \ref{thm:min-max} can be written as
\begin{align*}
	\bar V(t,x) = V(T-t,x) &=  \inf_{y} \sup_{c > c_{\bar H}}\{g(y) +  S^c(y,x) - c(T-t)\} \\ 
	&=  \inf_y \sup_{c > c_{\bar H}} \{g(y) + \bar S^c(x,y) - c(T-t)\}. 
\end{align*}

In general, it is not possible to interchange the $\inf$ and $\sup$ in the min-max representation as the following example shows. Note that the function $g$ does not satisfy the conditions of uniform continuity and boundedness. The example nonetheless illustrates what can go wrong when interchanging min and max. In Section \ref{sec:rare} this particular example is discussed further in the context of rare-event simulation.
\begin{exmp} \label{ex:maxmin}
Consider a one-dimensional terminal value problem, with Hamiltonian $\bar H(x,p) = \bar H( p) = p+\frac{1}{2}p^{2}$ and $g(x) = 0$ on $\partial (a,b)$ and $g(x) = \infty$ on $(a,b)$, where $a < 1 < b$ and $b-1 < 1-a$. The \Mane\ critical value is $c_{\bar H} = -1/2$ and the \Mane\ potential is given by
\begin{align*}
	\bar S^c(x,y) = \left\{\begin{array}{ll} (y-x)(-1+\sqrt{1+2c}), & y \geq x,\\
	(x-y)(1+\sqrt{1+2c}), &  y < x. \end{array} \right.
\end{align*}
By performing the optimization it follows that
\begin{align*}
	\sup_{c > c_{\bar H}}\{ \bar S^c(x,y) - c(T-t)\} =  \left\{\begin{array}{ll} \frac{T-t}{2}(\frac{y-x}{T-t} - 1)^{2}, & y \geq x,\\
	\frac{T-t}{2}(\frac{x-y}{T-t} - 1)^{2}, &  y < x. \end{array} \right.
\end{align*}
and, for $x < a$, we have
\begin{align*}
	\bar V (t,x) = \inf_{y \in \{a,b\}} \sup_{c > c_{\bar H}}\{ \bar S^c(x,y) - c(T-t)\} = \inf_{y \in \{a,b\}} \frac{T-t}{2}\Big(\frac{y-x}{T-t} - 1\Big)^{2}. 
\end{align*}
In particular, with $T=1$, we have 
\begin{align*}
	\bar V(0,0) =  \inf_{y \in \{a,b\}}\frac{1}{2}(y - 1)^{2} = \frac{1}{2}(b - 1)^{2}. 
\end{align*}
Consider interchanging the order of the inf and sup. For any $c> c_{\bar H}$ the infimum over the boundary is
\begin{align*}
	\inf_{y \in \{a,b\}}\ \{\bar S^c(0,y)-c\} = \left\{\begin{array}{ll} 
		a(-1+\sqrt{1+2c}) - c, & \text{for } c \geq 0,\\
		b(-1+\sqrt{1+2c}) - c, & \text{for } c < 0.\end{array} \right.
\end{align*} 
An elementary calculation shows that $\sup_{c > c_H} \inf_{y \in \{a,b\}} \{\bar S^c(0,y)-c\}$ is equal to
\begin{align*}
   \Big(\sup_{c \geq 0} \{a(-1+\sqrt{1+2c}) - c\}\Big) \vee \Big(\sup_{c<0}\{ b(-1+\sqrt{1+2c}) - c\}\Big) = 0. 
\end{align*}
We conclude that
\begin{align*}
	\bar V(0,0) = \inf_{y \in \{a,b\}} \sup_{c > c_{\bar H}}\{\bar S^c(0,y) - c\} >  \sup_{c > c_{\bar H}}\inf_{y \in \{a,b\}}  \{\bar S^c(0,y) - c\}. 
\end{align*}
\end{exmp}

\subsection{Problems on domains} 
 
 Let $\Omega \subset \R^n$ be an open domain, $g: \partial \Omega \to \R$ a uniformly continuous function representing the boundary condition and, for $(t,y) \in (0,\infty) \times \Omega$, let
\begin{align*}
	V(t,y) =  \inf_{\psi}\left\{ g(\psi(0)) + \int_{0}^{t} L(\psi(s), \dot \psi(s)) ds,   
	\psi(0) \in \partial \Omega, \psi(t) = y \right\}\!, 
\end{align*}
where the infimum is over all absolutely continuous functions $\psi: [0,\infty) \to \overline{\Omega}$, with $\psi(0) \in \partial \Omega$ and $\psi(t) \in \Omega$, $t > 0$. 
Then, $V$ is the unique continuous viscosity solution to
\begin{align*}
\begin{cases}
	 V_{t}(t,y) + H(y, DV(t,y)) = 0, & (t,y) \in (0,\infty) \times \Omega, \\
	V(0,y) = g(y), &  y \in \partial \Omega. 
\end{cases}
\end{align*}
The min-max representation is given by
\begin{align}
V(t,y)  &= \inf_{x \in \partial \Omega} \sup_{c > c_H} \{ g(x) + S^c(x,y) - ct\}.  \label{eq:min-maxdomain}
\end{align}

The terminal value problem on a domain $\Omega$ is
\begin{align*}
	\bar V(t,x) =  \inf\left\{ \int_{t}^{T} \bar L(\psi(s), \dot \psi(s)) ds + g(\psi(T)),   
	\psi(t) = x, \psi(T)  \in \partial \Omega \right\}, 
\end{align*}
where  $(t,x) \in [0,T) \times \Omega$.  
The function  $\bar V$ is the unique continuous viscosity solution to
\begin{align}\label{eq:termVD}
\begin{cases}
\bar V_{t}(t,x) - \bar H(x, -D\bar V(t,x)) = 0, & (t,x) \in [0,T) \times \Omega, \\
	\bar V(T,x) = g(x), & x \in \partial \Omega. 
\end{cases}
\end{align}
In this case the min-max representation is given by
\begin{align}
\bar V(t,x)  &= \inf_{y \in \partial \Omega} \sup_{c > c_{\bar H}} \{ g(y) + \bar S^c(x,y) - c(T-t)\} \label{eq:min-maxdomainterm}
\end{align}

\subsection{Exit from a domain}

 Let $\Omega \subset \R^n$ be an open domain, let $g: \partial \Omega \to \R$ be the boundary condition and take $T > 0$. Consider the minimal cost $\bar W$ of leaving the domain before time $T$, when starting from $(t,x) \in [0,T) \times \Omega$. The function $\bar W$ is given by
 \begin{align*}
	\bar W(t,x) &=  \inf_{\psi, \sigma}\left\{ \int_{t}^{\sigma} \bar L(\psi(s), \dot \psi(s)) ds + g(\psi(\sigma)),   
	\psi(t) = x, \psi(\sigma)  \in \partial \Omega \right\},
\end{align*}
where $t \leq \sigma \leq T$. By the change of variables, $ \tau = T-\sigma+t,$ and, for $t \leq s \leq T$, 
$\varphi(s) = \psi(t+s-\tau)$.
\begin{align*}
	\bar W(t,x) 
	&=  \inf_{\psi, t \leq \tau \leq T}\biggl\{ \int_{\tau}^{T} \bar L(\varphi(s), \dot \varphi(s)) ds + g(\varphi(T)),   
	\varphi(\tau) = x, \varphi(T)  \in \partial \Omega \biggr\} \\
	&= \inf_{t \leq \tau \leq T} \bar V(\tau, x), \quad (t,x) \in [0,T) \times \Omega, 
\end{align*}
with $\bar V$ as in \eqref{eq:termVD}.  $\bar W$ is the unique continuous viscosity solution to
\begin{align}
\label{eq:exitHJ}
\begin{cases}
	\bar W_{t}(t,x) - \bar H(x, -D\bar W(t,x)) = 0, & (t,x) \in [0,T) \times \Omega, \\
	\bar W(t,x) = g(x), & (t,x) \in [0,T] \times \partial \Omega. 
\end{cases}
\end{align}
In this case $\bar W$ can be represented as
\begin{align}
 \bar W(t,x)  &= \inf_{t \leq \tau \leq T}\inf_{y \in \partial \Omega} \sup_{c> c_{\bar H}}  \{ g(y) + \bar S^c(x,y) - c(T-\tau)\} \label{eq:min-maxexit}
\end{align}

Obviously $\bar W(t, x) \leq \bar V(t,x)$. If $c_{\bar H} \geq 0$, then we also have
\begin{align*}
 \bar W(t,x)  &= \inf_{t \leq \tau \leq T}\inf_{y \in \partial \Omega} \sup_{c> c_{\bar H}}  \{ g(y) + \bar S^c(x,y) - c(T-\tau)\} \\
 & \geq   \inf_{y \in \partial \Omega} \sup_{c > c_{\bar H}}  \inf_{t \leq \tau \leq T} \{ g(y) + \bar S^c(x,y) - c(T-\tau)\} \\
  & \geq    \inf_{y \in \partial \Omega}\sup_{c > c_{\bar H}} \{ g(y) + \bar S^c(x,y) - c(T-t)\}
  \\ &= \bar V(x,t).  	
\end{align*}
We have proved the following. 
\begin{prop}\label{prop:exit}
 If $c_{\bar H} \geq 0$, then $\bar W = \bar V$. 
 \end{prop}

\section{The Hopf-Lax-Oleinik representation}\label{sec:HopfLax}

Suppose the Hamiltonian $H$ is state-independent, that is, $H(x,p) = H(p)$. 
If $g$ is uniformly continuous, then the Hopf-Lax-Oleinik representation, see \cite[Ch. X]{Evans},  states that the function
\begin{align}\label{eq:HopfLax}
	V(t,y) = \inf_{x} \Big\{ g(x) + tL\Big(\frac{y-x}{t}\Big)\Big\},
\end{align}
is the unique continuous viscosity solution to 
\begin{align*}
\begin{cases}
	V_{t}(t,y) + H(DV(t,y)) = 0, & (t,y) \in (0,\infty) \times \R^{n}, \\
	V(0,y) = g(y), & y \in \R^{n}. 
\end{cases}
\end{align*}
We will demonstrate how the Hopf-Lax-Oleinik representation  follows directly from the min-max representation \eqref{eq:min-max}. 
\begin{prop} If $H$ is state-independent, then, for all $y \in \R^{n}$, 
\begin{align*}
V(t,y) = \inf_{x} \sup_{c > c_H} \{ g(x) + S^c(x,y) - ct\} =  \inf_{x} \Big\{ g(x) + tL\Big(\frac{y-x}{t}\Big)\Big\}.
\end{align*}
\end{prop}
\begin{proof}
 We begin by proving the inequality: for each $x$, 
\begin{align*}
\sup_{c > c_{H}} \{S^c(x,y) - ct \} \geq   tL\Big(\frac{y-x}{t}\Big). 
\end{align*}

Take $x \in \R^{n}$, $c> c_{H}$ and observe that for $p$ such that $H( p) = c$
\begin{align*}
	S^{c}(x,y) &= \inf _{\psi , t} \Big\{ \int_{0}^{t} H( p) + L(\dot \psi(s)) ds, \psi(0) = x, \psi(t) = y\Big\} \\
	&\geq \inf _{\psi, t} \Big\{ \int_{0}^{t} \langle p,\dot \psi(s)\rangle ds, \psi(0) = x, \psi(t) = y\Big\} \\
	&= \langle p,y-x \rangle. 
\end{align*}
It follows that
\begin{align*}
	S^{c}(x,y) - ct & \geq \sup_{p: H( p)= c} \{\langle p, y-x \rangle - tH( p)\} \\
	&= t \sup_{p: H( p)= c} \Big\{\langle p, \frac{y-x}{t} \rangle - H( p) \Big\}. 
\end{align*}
By Proposition \ref{prop:Fathi}, $S^{c}(x,y) = -\infty$ for $c < c_{H}$, which implies that the supremum over $c > c_{H}$ can be extended to the whole of $\R$. That is, 
\begin{align*}
	\sup_{c > c_{H}} \{ S^{c}(x,y) - ct\} &= \sup_{c \in \R} \{ S^{c}(x,y) - ct\}\\ & \geq  t \sup_{c \in \R}\sup_{p: H( p)= c} \Big\{\langle p, \frac{y-x}{t} \rangle - H( p)\Big\} \\ 
	&= tL\Big(\frac{y-x}{t}\Big).
\end{align*}

The reverse inequality 
\begin{align*}
\sup_{c > c_H} \{S^c(x,y) - ct\} \leq  tL\Big(\frac{y-x}{t}\Big),
\end{align*}
follows immediately by taking $\dot \psi(s) = (y-x)/t$ and observing that
\begin{align*}
	S^{c}(x,y) &\leq \int_{0}^{t} c + L(\dot \psi(s)) ds = \Big[c+L(\frac{y-x}{t}\Big)\Big] t.
\end{align*}
\end{proof}

\section{Applications in rare-event simulation}\label{sec:rare}
 
The simulation of rare events in stochastic models and the computation of their probabilities is a challenging problem with numerous applications in, for instance, biology, chemistry, engineering, finance, operations research, etc. In the rare-event setting the standard Monte Carlo algorithm fails because few particles will hit the relevant part of the state space, leading to a high relative error. There are several variance reduction techniques to improve computational efficiency that try to control the simulated particles in such a way that they reach the relevant part of the space. Such techniques can, if well designed, reduce the computational cost by several orders of magnitude. Examples of such techniques are importance sampling, multi-level splitting,  and genealogical particle methods. 

The common feature of all algorithms designed for the rare-event setting is that the control mechanism must be carefully chosen to control the relative error. Roughly speaking the large deviations of the stochastic model must be taken into account and guide the design of the algorithm. In a series of papers  \cite{DupuisWang, DW7, DSW, DLW3, VandenWeare} the authors have established the connection between efficient importance sampling algorithms and subsolutions to associated partial differential equations of Hamilton-Jacobi type that arise in large deviation theory. The results can be briefly summarized as follows. To compute an expectation of the form $E[\exp\{-n g(X^n(T))\}I\{X^n(T) \notin \Omega\}]$ the choice of sampling dynamics is associated with a control problem whose value function, in the rare-event limit, is given as the solution $\bar V$ to a Hamilton-Jacobi equation of the form \eqref{eq:termVD}. By constructing a (piecewise) classical subsolution $\bar U$ to \eqref{eq:termVD}, that is a piecewise $C^1(\bar \Omega)$ function $\bar U$ satisfying
\begin{align}
\begin{cases}
	\bar U _t (t,x) - \bar H(x, -D \bar U(t,x)) \geq 0, & (t,x) \in [0,T)\times \Omega, \\
	\bar U(T,x) \leq g(x), & x \in \partial \Omega,
\end{cases} \label{eq:simsub}
\end{align}
 the change of measure can be based on $D\bar U$ and the performance of the resulting algorithm is determined by the initial value $\bar U(0, x_0)$ of the subsolution. Asymptotically optimal performance is obtained if the value of the subsolution at the initial point $(0,x_0)$ coincides with that of the solution, $\bar U(0, x_0) = \bar V(0, x_{0})$. 

In multi-level splitting the situation is similar. In the most simple version of multi-level splitting the state space is partitioned into an increasing sequence of sets $C_{0} \subset C_{1} \subset \dots$ given as the level sets of an importance function $U$. A particle is simulated from an initial point $x_{0}$ and as it crosses over from, say, $C_{k+1}$ to $C_{k}$ for the first time, the particle generates a number of offsprings that are simulated independently of each other. Particles are killed if they reach a termination set. Each particle carries a weight that is updated at every split. By this procedure a random tree is produced, where each leaf is a particle that has either hit the set of interest or been killed. The sum of the weights of the particles that reach the target set is the estimate of the rare-event probability. The design of an efficient multi-level splitting algorithm relies on that the associated importance function is a certain multiple of a viscosity subsolution of the Hamilton-Jacobi equation associated with the large deviations of the system, see \cite{DupuisDean, DupuisDean2}. 

In what follows the emphasis will be on the construction of families of viscosity subsolutions associated with the min-max representation. To be precise, in this section the term viscosity subsolution refers to a function that satisfies the inequalities \eqref{eq:simsub} in the viscosity sense. For brevity the discussion is focused on terminal value problems, for exit problems everything is completely similar.

\subsection{Construction of subsolutions} 
The min-max representation \eqref{eq:min-maxdomainterm} provides at least two convenient ways to construct families of viscosity subsolutions, suitable for the design of rare-event simulation algorithms.

The most obvious way to construct viscosity subsolutions is perhaps to consider the family
\begin{align*}
	\bar U^{c}(t,x) = \inf_{y \in \partial \Omega} \{g(y) + \bar S^{c}(x,y) - c(T-t)\}, \quad c > c_{\bar H}. 
\end{align*}
The optimal choice of $c$ is to take $c$ as the maximizing energy level in the min-max-representation,
\begin{align}\label{eq:copt}
	\inf_{y \in \partial \Omega} \sup_{c > c_{\bar H}} \{g(y) + \bar S^{c}(x_0,y) - cT\}= g(y) + \bar S^{c}(x_0,y) - cT, 
\end{align} 
for the optimal pair $(c,y)$. Then $\bar U^c$ satisfies
\begin{align*}
\begin{cases} 
	\bar U^c_{t}(t,x) - \bar H(x, -D\bar U^c(t,x)) = 0, & (t,x) \in [0,T) \times \Omega, \\
	\bar U^c(T,x) \leq g(x), & x \in \partial \Omega,
\end{cases}
\end{align*}
that is,  $\bar U^c$ is a viscosity subsolution to \eqref{eq:termVD}. Note that $\bar U ^c$ is a subsolution to the exit problem as well if $c_{\bar H} \geq 0$. For either type of problem, let $K_1$ denote the loss in performance for simulation algorithms based on $\bar U ^c$, 
\begin{align*}
K_1 = \bar V (0,x_0) - \bar U ^c (0,x_0).
\end{align*}
If $K_1 = 0$ the subsolution $\bar U ^c$ gives rise to asymptotically optimal simulation algorithms.

The main obstacle when implementing an algorithm based on $\bar U^{c}$ is that the optimization over $y$ may be complicated and must be solved numerically at every $x$, leading to a significant overhead computational cost. 



A considerably simpler  family of viscosity subsolutions is given by
\begin{align*}
	\bar U^{c,y,K_2}(t,x) = g(y) + \bar S^{c}(x_{0},y) - \bar S^{c}(x_{0},x) - c(T-t) - K_2,  
\end{align*}
where $c > c_{\bar H}, y \in \partial \Omega, K_2 \geq 0$ and the constant $K_2$ must be chosen appropriately. Since $\bar S^{c}(x_{0}, x_{0}) = 0$ if follows that 
\begin{align*}
	\bar U^{c,y,K_2}(0,x_{0}) = g(y) + \bar S^{c}(x_{0},y) - cT - K_2,
\end{align*}
and the optimal choice of $(c,y)$ is such that 
\begin{align*}
	\bar U^{c,y,K_2}(0, x_{0}) = \inf_{y \in \partial \Omega} \sup_{c >  c_{\bar H}} \{g(y) + \bar S^{c}(x_{0},y) - cT - K_2\} = \bar V(0, x_{0}) - K_2. 
\end{align*} 
The function $\bar U^{c,y,K_2}$ satisfies
\begin{align*}
	\begin{cases} \bar U^{c,y,K_2}_{t}(t,x) - \bar H(x, -D\bar U^{c,y,K_2}(t,x)) = 0, & (t,x) \in [0,T) \times \Omega, \\
	\bar U^{c,y,K_2}(T,x) =  g(y) + \bar S^c(x_0,y) - \bar S^c(x_0,x) - K_2, & x \in \partial \Omega. 
	\end{cases}
\end{align*}
To satisfy the subsolution property at the terminal time, the boundary condition must be satisfied with inequality, i.e., it is required that $\bar U^{c,y,K_2}(T,x) \leq g(x)$ for each $x \in \partial \Omega$. It is therefore necessary to select 
\begin{align*}
	K_2 &=  \sup_{x \in \partial \Omega} \{g(y) + \bar S^c(x_0,y)-\bar S^c(x_0,x) - g(x)\}  \\&= 
	g(y) + \bar S^c(x_0,y) - \inf_{x \in \partial \Omega} \{g(x) + \bar S^c(x_0,x)\}.
\end{align*}
This shows why the constant $K_2$ must be included in the construction. 

If $\bar U^{c,y,K_2}$ is piecewise $C^{1} (\bar \Omega)$, then performance of the importance sampling algorithm based on $D\bar U^{c,y, K_2}$ is determined by $\bar V(0,x_0) - K_2$. That is, 
$K_2$ determines the loss in performance for simulation algorithms based on $\bar U ^{c,y,K_2}$; asymptotically optimal performance is achieved if $K_2=0$.  

Intuitively, it may seem as if an importance sampling algorithm based on the subsolution $\bar U ^{c}$ should have better performance than one based on $\bar U ^{c,y,K_2}$. However, as the following proposition shows, the two subsolutions actually have the same initial value and the performance of the corresponding simulation algorithms coincides (in the asymptotic sense). Moreover, we provide a sufficient condition for when asymptotic optimality holds.
\begin{proposition}
\label{prop:M}
	(i) The two subsolutions $\bar U ^c$ and $\bar U ^{c,y,K_2}$ have the same initial value, that is $K_1 = K_2 = K$. 
	
	(ii) A sufficient condition for $K=0$ is that there exists a saddle point $(c, y)$ for the min-max representation at the initial point $(0,x_0)$.
\end{proposition}
\begin{proof}
Proof of (i). Let  $(c,y)$ be the pair of energy level and boundary point that is optimal at $(0,x_0)$, 
	\begin{align*}
		\inf _{y \in \partial \Omega} \sup _{c > c_{\bar H}} \{ g(y) + \bar S ^{c} (x_0,y) - c T\} = g(y) + \bar S ^{c} (x_0,y) - c T.
	\end{align*}
	The difference between $\bar V (0, x_0)$ and $\bar U ^{c} (0,x_0)$ is
	\begin{align*}
		K_1 &= \bar V(0,x_0) - \bar U ^{c} (0,x_0) \\
		& = \inf _{y \in \partial \Omega} \sup _{c > c_{\bar H}} \{ g(y) + \bar S ^{c} (x_0,y) - c T\} - \inf _{y \in \partial \Omega} \{ g(y) + \bar S ^{c} (x_0, y) - c T\}.
	\end{align*}
	By the choice of $(c,y)$,
	\begin{align*}
		\bar V (0,x_0) = g(y) + \bar S ^{c} (x_0,y) - c T,
	\end{align*}
	and it follows that
	\begin{align*}
		K_1 &= g(y) + \bar S ^{c} (x_0,y) - c T - \inf _{y \in \partial \Omega} \{ g(y) + \bar S ^{c} (x_0, y) - c T \} \\
			& = g(y) + \bar S ^{c} (x_0,y) - \inf _{y \in \partial \Omega} \{ g(y) + \bar S ^{c} (x_0, y)\}.
	\end{align*}
	This is precisely the definition of $K_2$ and the proof of (i) is complete.
	
	Proof of (ii). 
%
	Let $f$ be defined as the function 
	\begin{align*}
		f(c,y) = g(y) + S^c (x_0,y) - cT.
	\end{align*}
	The maximal initial value is then $\bar V (0,x_0) = \inf _{y \in \partial \Omega} \sup_{ c > c_{\bar H}} f(c,y)$. 	Take $(c^*,y^*)$ to be a saddle point to the min-max representation at the initial point $(0,x_0)$,
	\begin{align*}
		\inf _{y \in \partial \Omega} \sup _{c > c_{\bar H}} f(c,y) & \leq f(c^*, y^*) \leq \sup _{c > c_{\bar H}} \inf _{y \in \partial \Omega} f(c,y).
	\end{align*}
	From (i) $K = K_1 = K_2$ satisfies
	\begin{align*}
		K &= g(y^*) + \bar S ^{c^*} (x_0,y^*) - \inf _{y \in \partial \Omega} \{ g(y) + \bar S ^{c^*} (x_0, y)\} \\
		& = g(y^*) + \bar S ^{c^*} (x_0,y) -c^* T - \inf _{y \in \partial \Omega} \{ g(y) + \bar S ^{c^*} (x_0, y) - c^* T\} \\
		& = f(c^*, y^*) - \inf _{y \in \partial \Omega} f(c^*,y).	
	\end{align*}
	The assumption that $(c^*, y^*)$ is a saddle point implies that
	\begin{align*}
		f(c^*, y^*) \leq f(c^*,y), \ \forall y \in \partial \Omega,
	\end{align*}
	and thus that $K \leq 0$. The reverse inequality, $K \geq 0$, is immediate and we conclude that $K = 0$.
%
%
%
\end{proof}
Although Proposition \ref{prop:M} shows that simulation algorithms based on the two subsolutions have the same asymptotic performance, the reasons for the potential loss in performance, $K$, are different in the two cases. For $\bar U ^c$, the terminal condition is guaranteed to hold due to the infimum over $\partial \Omega$. This however may cause a misspecification of the optimal boundary point $y$ at $(0,x_0)$, leading to a possible decrease in the initial value. For the second subsolution, $\bar U ^{c,y, K}$, the pair $(c,y)$ is chosen at the initial point and is therefore the optimal choice. However, the terminal condition is not guaranteed to hold and the constant $K$ must be included for this reason, causing a potential loss in performance. There are of course many situations in which $K=0$, which implies that $\bar U ^c$ and $\bar U^{c,y,0}$ are viscosity subsolutions and the associated rare-event simulation algorithms have asymptotically optimal performance. 

%
%

Before proceeding to some examples, let us point out the rather remarkable property that $D\bar U^{c,y,K}(t,x) = -D\bar S^{c}(x_{0}, x)$ does not depend explicitly on $y$ and therefore not on explicitly on the domain $\Omega$. This implies that, except for the optimal choice of $c$, chosen initially,  the way to change the measure in the importance sampling algorithm does not depend on the domain $\Omega$. The change of measure only has the effect to move away from the law of large numbers trajectories, but the change of measure does not take into account the shape of the domain. This class of subsolutions is particularly useful if the boundary of $\Omega$ is complicated. 

\subsection{Importance sampling  for small-noise diffusions}
In this section the construction of viscosity subsolutions outlined above is illustrated in the setting of small-noise diffusions. For simplicity we only consider one-dimensional examples. The theory for multi-dimensional diffusions is, of course similar, but the details are more involved. We comment on the multi-dimensional case at the end of this section. 

For the purpose of illustration, let $\{X^{\epsilon}(t); t \in [0,\infty)\}_{ \epsilon > 0}$ be a collection of one-dimensional diffusion processes such that, for each $\epsilon > 0$,  $X^{\epsilon}$ is the unique strong solution to the stochastic differential equation
\begin{align}\label{eq:smallnoisediffusion}
	dX^{\epsilon}(t) = b(X^{\epsilon}(t)) dt + \sqrt{\epsilon} \sigma(X^{\epsilon}(t)) dB(t), \quad X^{\epsilon}(0) = x_{0}, 
\end{align}
where $B$ is a Brownian motion and $b$, $\sigma$ are Lipschitz continuous and satisfy appropriate growth conditions so that a strong solution exists. In this example we take $b(x) = -D\Phi(x)$ where $\Phi$ is a potential function with a local minimum at $x_0$. 

Let $\Omega = (a,b)$ be an open set with $x_0 \in \Omega$ and define the stopping time $\tau^{\epsilon}$ as the first exist time of $\Omega$, $ \tau ^{\epsilon}= \inf\{t > 0: X^{\epsilon}(t) \in \partial \Omega\}$. We are interested in computing $\Prob ( \tau^{\epsilon} \leq T)$, the probability that the diffusion leaves the domain $\Omega$ before $T$.  

From the work of \cite{DupuisSpilio, VandenWeare} it follows that an importance sampling estimator for this quantity is based on sampling $X^{\epsilon}$ from a distribution $Q^{\epsilon}$ given by the Girsanov transformation
\begin{align*}
	\frac{dQ^\epsilon}{d\Prob} = \exp\Big\{ - \frac{1}{2\epsilon} \int_{0}^{T} \theta(t,X^{\epsilon}(t))^{2} dt + \frac{1}{\sqrt{\epsilon}} \int_{0}^{T} \theta(t, X^{\epsilon}(t)) dB(t)\Big\}, 
\end{align*}
where $\theta(t, x) = -\sigma(x) D\bar U(t,x)$ and $\bar U$ is a classical (or piecewise classical) subsolution to the Hamilton-Jacobi equation 
\begin{align*}
\begin{cases}
	\bar W_t(t,x) - \bar H(x,-D \bar W(t,x)) = 0, & (t,x) \in (0,T) \times \Omega, \\
	\bar W(t,x) = 0,  & (t,x) \in (0,T] \times \partial \Omega. 
\end{cases}
\end{align*}
Here the Hamiltonian $\bar H$ is given by 
\begin{align*}
	\bar H(x,p) = -D\Phi(x) p + \frac{1}{2}|\sigma(x) p|^{2}. 
\end{align*}
In this case $c_{\bar H} = 0$, see Example \ref{ex:diffusion},  and the \Mane\ potential can be computed as
\begin{align*}
	\bar S^{c}(x,y)  = \int_{x}^{y} \frac{1}{\sigma(z)}\Big(\frac{D\Phi(z)}{\sigma(z)} + \sign(z-x) \sqrt{\frac{D\Phi(z)^{2}}{\sigma^{2}(z)} + 2c} \;\Big)dz, \ x,y \in (a,b).
\end{align*}
To see this, recall that $y \mapsto \bar S^{c}(x,y)$ is a viscosity solution to $\bar H(y, DS(y)) = c$ at all $y \in (a,b)$, $y \neq x$, and note that all solutions $p$ to $\bar H(y,p(y)) = c$ are of the form
\begin{align}\label{eq:expr}
	p(y)= \frac{1}{\sigma(y)} \Big(\frac{D\Phi(y)}{\sigma(y)} \pm \sqrt{\frac{D\Phi(y)^{2}}{\sigma^{2}(y)} + 2c}\;\Big).
\end{align}
The \Mane\ potential $\bar S^{c}(x,\cdot)$ is a primitive function of $p$, and the maximal of all subsolutions vanishing at $x$, see Proposition \ref{prop:Fathi}(ii). Therefore the $\pm$ sign must be selected as $\sign(z-x)$.  


Let us explain the construction of the families $\{\bar U^c\}$ and $\{\bar U^{c,y,K}\}$ of viscosity subsolutions in this particular setting. 
Since $c_{\bar H} = 0$ it follows from Proposition \ref{prop:exit} that $\bar W = \bar V$ where $\bar V$ is the unique continuous viscosity solution to 
 \begin{align*}
 \begin{cases}
	\bar V_t(t,x) - \bar H(x,-D\bar V(t,x)) = 0, & (t,x) \in (0,T) \times \Omega, \\
	\bar V(T,x) = 0,  & x \in \partial \Omega. 
\end{cases}
\end{align*}
Consider first, with $c$ chosen at the initial point, 
\begin{align*}
	\bar U^{c}(t,x) &= \inf_{y \in \partial \Omega} \{\bar S^{c}(x,y) - c(T-t)\} \\
	&=  \bar S^{c}(x,a) \wedge \bar S^{c}(x,b) - c(T-t).
\end{align*}
Given the optimal choice of $c$, the corresponding change of measure is determined by
\begin{align*}
	\theta^{c}(t,x) &= -\sigma(x) D\bar U^{c}(t,x) \\
	&= 
	\left\{\begin{array}{rr}
		\frac{D\Phi(x)}{\sigma(x)} + \sqrt{\frac{D\Phi(x)^{2}}{\sigma^{2}(x)} + 2c}, & \text{if } \bar S^{c}(x,b) \leq \bar S^{c}(x,a), \\
		\frac{D\Phi(x)}{\sigma(x)} -\sqrt{\frac{D\Phi(x)^{2}}{\sigma^{2}(x)} + 2c}, & \text{if } \bar S^{c}(x,a) \leq \bar S^{c}(x,b), \end{array} \right. \\
		&= \frac{D\Phi(x)}{\sigma(x)} + \sign(\bar S^{c}(x,a) - \bar S^{c}(x,b)) \sqrt{\frac{D\Phi(x)^{2}}{\sigma^{2}(x)} + 2c}. 
\end{align*}
Next, consider
 \begin{align*}
	\bar U^{c,y, K}(t,x) &= \bar S^{c}(x_{0},y) - \bar S^{c}(x_{0}, x) - c(T-t) - K, 
\end{align*} 
where $y = a$ if $\bar S^{c(a)}(x_{0},a) - c(a)T < \bar S^{c(b)}(x_{0},b) - c(b)T$ and $y = b$ if the reverse inequality holds (we emphasize here the dependence of $c$ on $y$). Given the optimal value of $(c,y)$, the change of measure is given by
\begin{align*}
	 \theta^{c,y}(t,x) = -\sigma(x) D\bar U^{c,y,K}(x,t) &= 
	\frac{D\Phi(x)}{\sigma(x)} + \sign(x-x_{0}) \sqrt{\frac{D\Phi(x)^{2}}{\sigma^{2}(x)} + 2c}.  
\end{align*}

For the changes of measure, determined by $\theta^{c}$ and $\theta^{c,y}$, let us determine the corresponding dynamics under $Q^\epsilon$. 
By Girsanov's theorem, it holds that 
\begin{align*}
 B^{\epsilon} (t) = B(t) - \frac{1}{\sqrt{\epsilon}} \int _0 ^{t} \theta(s, X^{\epsilon}(s))ds
\end{align*}
is a $Q^{\epsilon}$-Brownian motion on $[0, \tau ^{\epsilon}]$ and $X^{\epsilon}$ satisfies $X ^{\epsilon}(0)=x_{0}$ and
\begin{align*}
	d X^{\epsilon} (t) &= -D\Phi(X^{\epsilon} (t))dt + \sigma(X^{\epsilon}(t))\theta(X^{\epsilon}(t)) dt + \sqrt{\epsilon} \sigma(X^{\epsilon}(t) dB^{\epsilon}(t) \\
	&= \mu(X^\epsilon(t)) dt + \sqrt{\epsilon} \sigma(X^{\epsilon}(t)) dB^\epsilon(t), 
	\end{align*}
	where the drift $\mu (X^{\epsilon}(t)$ is given by
	\begin{align*} 
	\mu (X^{\epsilon}(t) = \sign(\bar S^{c}(X^{\epsilon}(t), a) - \bar S^{c}(X^{\epsilon}(t), b))\sqrt{D\Phi(X^{\epsilon}(t))^{2} + 2c\sigma^{2}(X^{\epsilon}(t))},
\end{align*}
if $\theta = \theta^{c}$, and 
\begin{align*}
	\mu (X^{\epsilon}(t) = \sign(X^{\epsilon}(t)-x_{0})  \sqrt{D\Phi(X^{\epsilon}(t))^{2}+ 2c\sigma^{2}(X^{\epsilon}(t))},
\end{align*}
if $\theta = \theta^{c,y}$.
\begin{exmp}[Numerical illustration]
Consider computing the probability $P (\tau^{\epsilon} \leq T )$ for a diffusion with a double-well potential given by $\Phi(x) = \frac{1}{2}(x^{2}-1)^{2}$. Take $\sigma(x) = 1$, $\Omega = (-1.42, 1.42)$ and $x_0 =1$. Estimates of the probability $P (\tau ^{\epsilon} \leq T )$ and corresponding relative errors for different values of $\epsilon, \ T$ are shown in Table \ref{table:doubleWell2}. The estimates and relative errors were computed over $50$ batches of $N=10^4$ samples each and with a time discretization of $T \times 10^{-3}$; the subsolution based on $\theta ^{c,y}$ was used to define the sampling dynamics.
\begin{table}[ht]
\caption{Estimates of $P(\tau ^{\epsilon} \leq T)$ and corresponding relative errors; $\Omega = (-1.42,1.42)$, $x_0 = 1$.}
\label{table:doubleWell2}
\centering
\resizebox{\columnwidth}{!}{%
\begin{tabular}{c | c | c | c | c| c| c| c| c}
\hline \hline
 & \multicolumn{2}{c|}{$T=0.25$} & \multicolumn{2}{c|}{{$T=0.5$}} & \multicolumn{2}{c|}{{$T=1$}} & \multicolumn{2}{c}{{$T=2$}} \\
\hline
 $\epsilon$ & Est. & Rel. err. & Est. & Rel. err. & Est. & Rel. err. & Est. & Rel. err. \\ 
 \hline
 0.09 & 3.898e-6 & 0.0254 & 2.373e-5 & 0.0174 & 6.717e-5 & 0.0253 & 1.599e-4 & 0.154 \\
 0.05 & 1.922e-10 & 0.0308 & 2.457e-9 & 0.0233 & 8.641e-9 & 0.0325 & 2.276e-8 & 0.185  \\
 0.03 & 6.876e-17 & 0.0332 & 2.424e-15 & 0.0296 & 1.098e-14 & 0.0437 & 3.469e-14 & 0.256  \\
 \hline
\end{tabular}
}
\end{table}
\end{exmp}
Before proceeding it must be noted that, although the performance of the algorithm is very good in the previous example, our construction of subsolutions does not address the problems of diminishing performance that may arise when the time horizon is large, as reported and treated in \cite{DupuisSpilio}. This is somewhat hinted at in Table \ref{table:doubleWell2} for $T=2$. In fact, for large $T$ the optimal energy level $c$ will approach $c_H = 0$ and the \Mane\ potential will approach the Freidlin-Wentzell quasi potential. 

In situations where the inf and the sup in the min-max representation cannot be interchanged importance sampling algorithms based on $\bar U^{c,y,K}$ may have poor performance. This is illustrated in the following toy problem, which is closely related to \cite[Sec.\ 3.4, Ex.\ 1]{DupuisWang}.
\begin{exmp} Let $\{X^{\epsilon}(t); t \in [0,\infty)\}_{ \epsilon > 0}$ satisfy $X^\epsilon(0) = 0$ and 
\begin{align*}
	dX^\epsilon(t) = dt + \sqrt{\epsilon} dB(t). 
\end{align*}
We are interested in constructing an efficient rare-event simulation algorithm for computing $\Prob (X^\epsilon(1) \notin (a,b) )$. 
The associated Hamiltonian is $\bar H(x,p) = \bar H(p) = p+\frac{1}{2}p^2$, which is the Hamiltonian encountered in Example \ref{ex:maxmin}. With $a$ and $b$ as in  Example \ref{ex:maxmin} it follows that 
\begin{align*}
	\bar U^{c,y,K}(t,x) = \bar S^c(0,y) - \bar S^c(0,x) - c(1-t) - K,
\end{align*}
with the optimal choice of $c$ and $y$ being $(c,y) = ((b^2-1)/2, b)$.  The change of measure based on $\bar U^{c,y,K}$ is given by $\theta^{c,y} = D\bar S^c(0,x) = \sign (x) b-1$ and the resulting dynamics under $Q^\epsilon$ is
\begin{align*}
	dX^\epsilon(t) = \sign(X^{\epsilon}(t))b dt +  \sqrt{\epsilon} dB^\epsilon(t). 
\end{align*}
The performance of the algorithm based on $\bar U^{c,y,K}$ is determined by the initial value $\bar U^{c,y,K}(0,0) = \bar V(0,0) - K$, where
\begin{align*}
	K = \bar S^c(0,b)-\bar S^c(0,a) = b(-1+b) - a(-1+b) = (b-a)(b-1). 
\end{align*}
We conclude that if $b-a$ is large, then the performance of the algorithm may be poor. 

\end{exmp}

In this section the construction of appropriate subsolutions in the context of small-noise diffusions has been illustrated in the one-dimensional setting. The multi-dimensional setting is more challenging. In particular, the computation of the \Mane\ potential is more involved. Since $y \mapsto \bar S^c(x,y)$ is a viscosity solution to the stationary Hamilton-Jacobi equation it follows that its gradient must be of the form $p$ where $p$ solves 
\begin{align*}
	c = \bar H(y,p(y)) = \langle -D\Phi(y), p(y)\rangle + \frac{1}{2}|\sigma(y)p(y)|^2. 
\end{align*}
In some cases the gradient $p$ of the \Mane\ potential can be found via the method of characteristics, see e.g.\ \cite{Evans}. The theory outlined in this paper shows that if the gradient of the \Mane\ potential or the \Mane\ potential itself can be found, then efficient rare-event simulation algorithms can be constructed, but it does not provide answers in situations where they are difficult to find.

\subsection{Importance sampling for birth-and-death processes}
Consider a collection $\{X^n (t); t \in [0,T ]\}_{n\geq 1}$ of one-dimensional continuous-time birth-and-death processes on $\N/n$, starting at $X^n_{0} = x_0$, having birth rates $n \lambda(x)$ and death rates $n \mu(x)$. Here $\lambda, \mu: \R \to [0,\infty)$ are assumed to bounded and Lipschitz continuous. 
The infinitesimal generator $\mathcal{A}^n$ of $X^n$ is given by
\begin{align*}
	\mathcal{A}^n f(x) = n \lambda(x)\Big(f(x+n^{-1})-f(x)\Big) + n\mu(x)\Big(f(x-n^{-1})-f(x)\Big). 
\end{align*}

Take an open interval $(a,b) \subset \R$ with $x_{0} \in (a,b)$ and denote the exit time of $\Omega$ by $\tau ^n = \inf\{t \geq 0: X^n(t) \notin \Omega\}$. We are interested in computing the exit probability $P(\tau^n \leq T)$, for some fixed $T > 0$.

The Hamiltonian associated with the birth-and-death process is given by  
\begin{align*}
\bar H(x,p) = \mu(x)(e^{-p}-1)+\lambda(x)(e^p -1). 
\end{align*}
For simplicity,  we make the additional assumption that $\mu(x_{0}) = \lambda(x_{0})$, so that the \Mane\ critical value is $c_{\bar H} = 0$, see Example \ref{ex:birthdeath}. 

Under the stated assumptions, the sequence $\{ X^n \}$ satisfies the large deviation principle in $\mathcal{D} ([0,T]; \R)$ with rate function
\begin{align*}
	I(\psi) = \int_0^T \bar L(\psi(s), \dot \psi(s)) ds, \quad \psi(0) = x_0, 
\end{align*}
where $\psi$ is absolutely continuous and $\bar L(x,v) = \sup_p\{ pv-\bar H(x,p)\}$ is the local rate function, see e.g.\ \cite{FengKurtz, ShwartzWeiss}.

Similarly to the work of \cite{DLW3, DWJackson} it follows that an importance sampling estimator for the exit probability is obtained by sampling $X^{n}$ independently from a distribution $Q^{n}$ with $P \ll Q^{n}$ and take the estimator as the sample mean of
\begin{align*}
	\frac{dP}{dQ ^n}(X^n) I\{ \tau ^{n}\leq T\}. 	
\end{align*}
The sampling distribution,  $Q^n$,  is a probability measure, parametrized by $\theta^n$,  such that $X^n$, under $Q^{n}$,  is a birth-and-death process with birth and death rates given by
\begin{align*}
	\lambda^{Q^n}(x) = \lambda(x)e^{-\theta^n(x)}, \; \mu^{Q^n}(x) = \mu(x)e^{\theta^n(x)}. 
\end{align*} 

An efficient estimator is obtained by taking $\theta^{n} = -D\bar U$ where $\bar U$ is a
 classical (or piecewise classical) subsolution of
\begin{align*}
\begin{cases}
	\bar W_t(t,x) - \bar H(x,-D \bar W(t,x)) = 0, & (t,x) \in (0,T) \times \Omega, \\
	 \bar W(t,x) = 0, & (t,x) \in (0,T] \times \partial \Omega, 
\end{cases}
\end{align*}
with the property that $\bar U(0,x_{0}) = \bar W(0,x_{0})$. Since $c_{\bar H} = 0$ it follows from Proposition \ref{prop:exit} that $\bar W = \bar V$ where $\bar V$ is the solution to the terminal value problem 
\begin{align*}
\begin{cases}
	\bar V_t(t,x) - \bar H(x,-D \bar V(t,x)) = 0, & (t,x) \in (0,T) \times \Omega, \\
	\bar V(T,x) = 0,  & x \notin \Omega.  
\end{cases}
\end{align*}

In this example the function 
\begin{align*}
	p^c(y) = \log\left[\frac{c+\lambda(y)+\mu(y)}{2\lambda(y)}\pm\sqrt{\Big(\frac{c+\lambda(y)+ \mu(y)}{2\lambda(y)}\Big)^2-\frac{\mu(y)}{\lambda(y)} }\;\right], 
\end{align*}
is the solution to $\bar H(y,p^c(y)) = c$. The \Mane\ potential $y \mapsto \bar S^{c}(x,y)$ is a primitive function of $p^c$, and the maximal of all viscosity subsolutions vanishing at $x$, see Proposition \ref{prop:Fathi}. Therefore the $\pm$ sign must be taken as positive for trajectories to the right, $y > x$, and negative for trajectories to the left, $y < x$.  Consequently, the \Mane\ potential is given by
\begin{align*}
	\bar S^c(x,y) &= \int_x^y \log\Biggl[\frac{c+\lambda(z)+\mu(z)}{2\lambda(z)} +\sign(z-x) \sqrt{\Big(\frac{c+\lambda(z)+\mu(z)}{2\lambda(z)}\Big)^2-\frac{\mu(z)}{\lambda(z)} }\;\Biggr] dz. 
\end{align*}

Let us illustrate the two families $\{\bar U^c\}$ and $\{\bar U^{c,y,K}\}$ of viscosity subsolutions in this setting. Consider first
\begin{align*}
	\bar U^{c}(t,x) &= \inf_{y \in \partial \Omega}\{\bar  S^{c}(x,y) - c(T-t)\} \\
	&= \bar S^{c}(x,a) \wedge \bar S^{c}(x,b) - c(T-t). 
\end{align*}
Given the optimal choice of $c$ the new birth and death rates are given by
\begin{align*}
	\lambda^{Q^{n}}(x) &= \lambda(x) \Biggr[\frac{c+\lambda(x)+\mu(x)}{2\lambda(x)} +\sign(\bar S^{c}(x,a)-\bar S^{c}(x,b)) \sqrt{\Big(\frac{c+\lambda(x)+\mu(x)}{2\lambda(x)}\Big)^2-\frac{\mu(x)}{\lambda(x)} }\;\Biggr],
\end{align*}
\begin{align*}
	\mu^{Q^{n}}(x) &= \mu(x) \Biggl[\frac{c+\lambda(x)+\mu(x)}{2\lambda(x)} +\sign(\bar S^{c}(x,a)-\bar S^{c}(x,b)) \sqrt{\Big(\frac{c+\lambda(x)+\mu(x)}{2\lambda(x)}\Big)^2-\frac{\mu(x)}{\lambda(x)} }\;\Biggr]^{-1}\!\!\!\!\!.
\end{align*}


Next, consider $\bar U^{c,y,K}$ given by
\begin{align*}
	\bar U^{c,y,K}(t,x) = \bar  S^{c}(x_{0},y) - \bar S^{c}(x_{0},x) - c(T-t) - K. 
\end{align*}
Given the optimal choice of $c$ the new birth and death rates are given by
\begin{align*}
	\lambda^{Q^{n}}(x) &= \lambda(x) \Biggl[\frac{c+\lambda(x)+\mu(x)}{2\lambda(x)} +\sign(x-x_{0}) \sqrt{\Big(\frac{c+\lambda(x)+\mu(x)}{2\lambda(x)}\Big)^2-\frac{\mu(x)}{\lambda(x)} }\;\Biggr],
\end{align*}
\begin{align*}
\mu^{Q^{n}}(x) &= \mu(x) \Biggl[\frac{c+\lambda(x)+\mu(x)}{2\lambda(x)} +\sign(x- x_{0}) \sqrt{\Big(\frac{c+\lambda(x)+\mu(x)}{2\lambda(x)}\Big)^2-\frac{\mu(x)}{\lambda(x)} }\;\Biggr]^{-1}\!\!\!\!\!.
\end{align*}

\begin{exmp}[Numerical illustration]
Consider a birth-and-death process $X^n$ with rates $\lambda(x) = \rho x (1-x)$, some $\rho >0$, and $\mu (x)=x$. The process $X^n$ can be thought of as the ratio of infected individuals in a population of size $n$ where infected individuals immediately upon recovery are again susceptible (the SIS model). Table \ref{table:birthDeath1} shows estimates of the probability $P( \tau ^n \leq T )$, and corresponding relative errors, for $\rho = 3$, $\Omega = (1/2, 5/6)$, $x_0 = 2/3$ and $T=1/2$. All estimates and relative errors were computed over $50$ batches of $N = 10^3$ samples each.
\begin{table}[ht]
\caption{Estimates and relative errors of $\Prob(\tau ^n \leq T)$ for a birth-and-death process with rates $\lambda(x) =3x(1-x)$, $\mu(x) = x$, $ \Omega = (1/2, 5/6)$, $x_0 = 2/3$ and $T=1/2$.}
\label{table:birthDeath1}
\centering
\begin{tabular}{c | c | c}
\hline \hline
 $n$ & Est. & Rel. err.\\ 
 \hline
100 & 7.806e-3 & 0.0438 \\
200 & 5.289e-5 & 0.0512\\
300 & 4.421e-7 & 0.0732\\
400 & 6.891e-9 & 0.0736\\
500 & 6.479e-11 & 0.101\\
 \hline
\end{tabular}
\end{table}
\end{exmp}
Table \ref{table:birthDeath1} illustrates good performance of the proposed importance sampling algorithm as $n$ increases. 

\section*{Appendix}

\begin{proof}[Proof of Proposition \ref{prop:W}]
Take $x \in \R^n$.  First we prove the viscosity subsolution property. 
Suppose that $M(\cdot\, ;  x)-v$ has a local maximum at $(t_0, y_0)$ and, contrary to what we want to show, that $v_t(t,y) + H(y,Dv(t,y)) \geq \theta > 0$ for $|t-t_0| + |y-y_0| \leq \delta$, for some $\delta > 0$. We may assume that $\delta$ is sufficiently small that 
\begin{align*}
	M(t,y; x) - v(t,y) \leq M(t_0, y_0; x) - v(t_0, y_0), \quad \text{for } |t-t_0| + |y-y_0| \leq \delta.
\end{align*}
Take any $h > 0$ and $y$ with $h+|y-y_0| \leq \delta$ and consider any trajectory $\psi$ such that $\psi(0) = y$, $\psi(h) = y_0$ and $|\psi(s) - y_0| \leq \delta$ for all $s \in [0,h]$. By optimality and the last inequality
\begin{align*}
	0 &\geq M(t_0,y_0; x) - M(t_0-h,y; x) - \int _{0} ^ {h} L(\psi(s), \dot \psi(s))ds \\
	& \geq v(t_0,y_0) - v(t_0-h,y) - \int _{0} ^ {h} L(\psi(s), \dot \psi(s))ds \\
	&= \int _{0} ^{h} \frac{d}{ds} v (s,\psi (s)) - L(\psi(s), \dot \psi(s))ds \\
	& = \int _{0} ^{h} v_t(s,\psi(s)) + \langle Dv(\psi(s)), \dot \psi (s) \rangle - L(\psi(s), \dot \psi(s))  ds.
\end{align*}
We may assume that $h$ and $\dot \psi$ are chosen such that, using the conjugacy between $H$ and $L$,
\begin{align*}
	H(\psi(s) , Dv(s, \psi(s))) \leq \langle Dv(s, \psi(s)), \dot \psi(s) \rangle - L(\psi(s), \dot \psi(s)) + \frac{\theta h}{2},
\end{align*}
for all $s\in [0,h]$. Then
\begin{align*}
	\frac{\theta h}{2} \geq \int _{0} ^{h} v_t(s,\psi(s)) + H(\psi(s), Dv(s,\psi(s))) ds \geq \theta h,
\end{align*}
which is a contradiction. Thus, it must indeed hold that
 $$v_t(t_0, y_0) + H(y_0, Dv(t_0, y_0)) \leq 0.$$

Next we prove the supersolution property on $\R^n \setminus \{x\}$. Suppose $M(\cdot\, ; x)-v$ has a local minimum at $(t_0, y_0)$ with $y_0 \neq x$ and, contrary to what we want to show, that $v_t(t,y) +H(y,Dv(t,y)) \leq -\theta < 0$ for $|t-t_0| + |y-y_0| \leq \delta$, for some $\delta > 0$. We may assume that $\delta$ is sufficiently small that $|x-y_0| > \delta$ and
\begin{align*}
	M(t,y; x) - v(t,y) \geq M(t_0,y_0; x) - v(t_0, y_0), \quad \text{for } |t-t_0| +  |y-y_0| \leq \delta.
\end{align*}
By optimality we may select $h> 0$ and $y$ with $h+|y-y_0| \leq \delta$ and a trajectory $\psi$ such that $\psi(0) = y$, $\psi(h) = y_0$ and $s+|\psi(s) - y_0| \leq \delta$ for all $s \in [0,h]$, with the property that
\begin{align*}
	M(t_0, y_0; x) \geq M(t_0-h, y; x) + \int_0^h L(\psi(s), \dot \psi(s)) ds - \frac{\theta h}{2}.
\end{align*}
The last inequality implies
\begin{align*}
	\frac{\theta h}{2} &\geq M(t_0-h,y; x) - M(t_0,y_0; x) + \int _{0} ^ {h} L(\psi(s), \dot \psi(s))ds \\
	& \geq v(t_0-h, y) - v(t_0, y_0) + \int _{0} ^ {h} L(\psi(s), \dot \psi(s))ds \\
	&= \int _{0} ^{h} - \frac{d}{ds} v (s,\psi (s)) + L(\psi(s), \dot \psi(s))ds \\
	& = \int _{0} ^{h} - \Big(v_t(s,\psi(s)) + \langle Dv(\psi(s)), \dot \psi (s) \rangle - L(\psi(s), \dot \psi(s))  
	\Big) ds \\
	&\geq \int_0^h - \Big( v_t(s,\psi(s)) + H(\psi(s), Dv(\psi(s))) \Big) ds. 
\end{align*}
We conclude that
\begin{align*}
	- \frac{\theta h}{2} \leq \int _{0} ^{h} v_t(s,\psi(s)) + H(\psi(s), Dv(\psi(s))) ds \leq  -\theta h,
\end{align*}
which is a contradiction. Thus, it must indeed hold that 
$$v_t(t_0, y_0) + H(y_0, Dv(t_0, y_0)) \geq 0.$$
This completes the proof of the first claim.

The proof of the second statement is completely analogous to the proof of Proposition \ref{prop:Fathi}(ii) and is therefore omitted.
\end{proof}

\bibliographystyle{plain}
\bibliography{references}

\end{document}